\documentclass[11pt]{article}
\usepackage[a4paper,margin=1in]{geometry}
\usepackage[T1]{fontenc}
\usepackage{amsmath,amssymb,amsthm,mathtools}
\usepackage{newtxtext,newtxmath}
\usepackage{graphicx,booktabs,enumitem,setspace,microtype,needspace,float}
\usepackage[authoryear]{natbib}
\usepackage{authblk}
\usepackage[dvipsnames]{xcolor}
\usepackage[colorlinks=true, allcolors=black]{hyperref}
\allowdisplaybreaks
\theoremstyle{plain}
\newtheorem{theorem}{Theorem}[section]
\newtheorem{lemma}[theorem]{Lemma}
\newtheorem{proposition}[theorem]{Proposition}

\theoremstyle{definition}
\newtheorem{definition}[theorem]{Definition}
\newtheorem{assumption}{Assumption}[section]
\theoremstyle{remark}
\newtheorem{remark}[theorem]{Remark}
\numberwithin{equation}{section}
\newcommand{\E}{\mathbb E}
\newcommand{\Pp}{\mathbb P}
\newcommand{\R}{\mathbb R}
\newcommand{\F}{\mathcal F}
\newcommand{\A}{\mathcal A}

\newcommand{\Tset}{\mathcal T}
\newcommand{\ind}{\mathbf 1}
\newcommand{\dd}{\,\mathrm d}
\title{\textbf{Optimal Dividends and Proactive Capital Injection over a Finite Horizon}}
\ifdefined\anonymize
\author{}
\else

\author[a]{Ran Xu\thanks{\texttt{ran.xu@xjtlu.edu.cn}}}
\author[b]{Zuo Quan Xu\thanks{\texttt{maxu@polyu.edu.hk}}}
\author[a]{Xueying Zhang\thanks{\texttt{xueying.zhang22@student.xjtlu.edu.cn}}}
\affil[a]{Department of Financial and Actuarial Mathematics, Xi'an Jiaotong-Liverpool University, Suzhou, China}
\affil[b]{Department of Applied Mathematics, The Hong Kong Polytechnic University, Hong Kong, China}
\fi

\begin{document}
\maketitle
\begin{abstract}
We study a finite-horizon dividend and capital-injection problem for a Brownian surplus process with time-varying financing costs. Under a classical passive-injection policy, capital is supplied only at the solvency boundary. Anticipated increases in financing costs may instead make it optimal to inject capital while reserves remain strictly positive. We characterize this proactive financing mechanism through an auxiliary Dynkin game for the marginal value of liquidity. The game identifies an endogenous lower financing boundary and an upper dividend boundary, and yields restrictions on intervention timing and comparative statics with respect to the financing-cost path. Under a regular intervention-band condition, we recover the control value by spatial integration of the game value and verify that the optimal policy is two-sided reflection between the two boundaries. We establish this condition for a canonical class with
decreasing marginal dividend rewards, increasing marginal financing costs, and a compatible piecewise affine terminal marginal value. Strictly decreasing financing costs eliminate positive-reserve financing contact and reduce the marginal game to a one-player stopping problem. Numerical examples illustrate the optimal policies and the effects of anticipated financing conditions on shareholder value and intervention thresholds.
\end{abstract}
\noindent\textbf{Keywords:} optimal dividends, capital injection, singular control, Dynkin game, finite horizon.

\noindent\textbf{MSC (2020):} 93E20, 60G40, 91G05.

\section{Introduction}

The optimal dividend problem is a classic topic in actuarial science, financial mathematics, and stochastic optimal control. Originally developed by \citet{DeFinetti1957}, the basic formulation asks how a firm should distribute its surplus to shareholders to maximize the expected discounted dividends before ruin. Optimal policies frequently have a barrier or threshold structure: liquid reserves are retained below an intervention threshold, and any excess above it is distributed
as dividends. Diffusion formulations are developed by \citet{AT97} and \citet{Hojgaard1999controlling}. Related L\'evy models are studied by \citet{Avram07}. For further work on optimal dividends, see, among others, \citet{schmidli2007stochastic}, \citet{Avanzi-2007}, and \citet{azcue2012optimal}. These results make the dividend problem a natural setting for studying a firm's liquidity allocation over time.

However, in the classical dividend model, a sufficiently aggressive payout policy
eventually exposes the firm to ruin. This limitation has motivated an
extensive literature in which shareholders are allowed or required to
inject capital in order to maintain solvency. The resulting control problem
must balance the immediate benefit of dividend distributions against the
future cost of recapitalizing the firm. In infinite-horizon,
time-homogeneous models, the optimal policy is frequently characterized by
an upper dividend barrier together with capital injection at the solvency
boundary. Under this policy, reserves are prevented from becoming negative,
but new capital is supplied only when the firm has exhausted its existing
liquidity. Capital injection is therefore reactive: it serves as a
last-resort mechanism for avoiding ruin rather than as an independently
timed financing decision. Dividend and capital-injection problems of this
type have been studied under a variety of surplus dynamics, transaction-cost
specifications, and solvency requirements; see, for example,
\cite{KS08,SS11,Zhu16,Yao-2011,Zhao-2014,Lindensjo-2020, xu2020optimal,wang2025optimal}.

The corresponding finite-horizon problem is substantially more delicate.
Once a fixed terminal date is introduced, time becomes an additional state
variable, the intervention boundaries generally depend on the remaining time to maturity, and the treatment of the terminal payoff becomes part of the control problem. For the pure-dividend problem without capital injection,
\citet{DeAngelis2017} characterized the optimal strategy through reflection
at a moving boundary and established a probabilistic connection between the
singular control problem and an associated optimal stopping problem. This
connection provides an alternative to a direct analysis of the
time-dependent Hamilton--Jacobi--Bellman equation and, in particular,
identifies the dividend boundary through the stopping region of the
marginal problem.

The closest benchmark for the present paper is
\citet{ferrari2019optimal}, who study a finite-horizon dividend problem with
capital injection for a drifted Brownian surplus process. Their formulation
allows both the marginal reward from dividend payments and the marginal
cost of capital injections to vary deterministically over time. By applying
the integration approach developed in \citet{karoui2006integration}, they
connect the singular control problem to a one-player optimal stopping problem. Their model prescribes capital injection as the minimal regulator at zero and optimizes the dividend process. Their Remark~2.4 separately gives sufficient conditions under which this injection rule is optimal in a broader solvency-preserving dividend control problem.

Related finite-horizon dividend,
reinsurance and capital injection problems have been considered in more
recent studies, including
\cite{chen2021free,guan2023finite,han2025finite}. These models differ in
their objectives and technical formulations, but largely preserve the
classical interpretation of capital injection as an intervention triggered
by imminent insolvency.

The present paper starts from the observation that waiting until reserves
are exhausted need not be optimal for capital injection when financing conditions vary over time.
In particular, external equity issuance need not become progressively cheaper over
the planning horizon. If shareholders anticipate that obtaining capital
will become more expensive, they may prefer to inject funds while current
financing conditions remain favorable, even when the firm still holds
positive reserves. Such a capital injection does not respond to an immediate
solvency crisis. Instead, it increases current liquidity in anticipation of
future financing needs and higher future recapitalization costs.

Motivated by this possibility, we study a finite-horizon dividend and
capital-injection problem for a Brownian surplus process with deterministic,
time-dependent marginal dividend rewards and external financing costs. In contrast
to models in which capital injection is activated only at the solvency boundary, we
allow shareholders to choose both the dividend and capital-injection processes, rather than prescribing capital injection
as the minimal regulator at zero. We interpret the
reserve as liquid funds available for operations and shareholder
distributions. A capital injection made when the reserve is strictly positive is
therefore a \emph{proactive capital injection}. The central questions are
when this additional financing flexibility has value, under what conditions
a proactive injection can occur, and how its availability changes the joint
timing of financing and dividend distributions.

These questions are closely related to the corporate finance literature on
liquidity management under costly external finance. In the dynamic models
of \citet{decamps2011free} and \citet{bolton2011unified}, corporate cash
holdings, shareholder payouts, and external financing are jointly
determined by the marginal value of internal liquidity. Financing frictions
can make one unit of cash retained within the firm more valuable than the
same unit distributed immediately to shareholders. They therefore create
an intertemporal tradeoff between current payouts and the preservation of
liquidity for future investment or funding needs. Our model isolates a
closely related tradeoff in a tractable singular control setting. Because
both dividend payout and financing are represented as (singular) controls, their timing can be
described through explicit intervention regions, and the marginal value of
liquidity determines whether the firm distributes, retains, or raises
funds.

The specification of the financing cost is also motivated by empirical
evidence that raising external equity is costly and that the magnitude of
the associated frictions varies across firms and over time.
\citet{hennessy2007costly} obtain economically substantial estimates of the
costs of external equity financing from a structural corporate-finance model.
\citet{lee2009seasoned} find that poorer quality accounting information is
associated with higher underwriting fees, more negative announcement returns, and a higher probability that a seasoned equity offer is
withdrawn. These findings support a distinction between the amount of
liquidity provided to the firm and the total cost borne by shareholders to
provide that liquidity. In particular, they caution against interpreting
the coefficient attached to capital injection as a discount factor alone.
In our formulation, the marginal financing coefficient represents the
shareholder cost of increasing the firm's liquid reserve and may incorporate
issuance costs, information frictions, or other financing wedges.

Allowing this coefficient to vary with time is particularly relevant over
a finite horizon. \citet{choe1993common} document the business-cycle variation
in equity issuance and find less negative announcement effects during
expansions, consistent with lower adverse-selection costs in those periods.
\citet{belo2019external} study aggregate shocks to external equity financing
costs and their implications for financing flows, investment, and asset
prices. \citet{chan2021policy} find that greater economic policy uncertainty
is associated with larger offer-price discounts in seasoned equity offerings, especially
when economic conditions are weak. Taken together, these studies motivate
allowing the marginal cost of capital to change flexibly with calendar
time, rather than imposing a profile under which future financing is always
more attractive. They do not, of course, imply that a firm can forecast
future financing conditions with certainty. The deterministic financing cost
path in our model should instead be understood as a reduced-form
representation of an anticipated financing environment over a specified
planning horizon.

Financing decisions may also reflect funding needs that extend beyond the
instant in which securities are issued. \citet{huang2021cash} find that
equity issuers retain a substantial proportion of issuance proceeds as cash
and that anticipated near-future funding needs help explain this retention.
Their evidence supports analyzing current financing decisions jointly with
the prospective value of corporate liquidity. In particular, external funds may be raised before
all existing liquidity is exhausted because the value of cash depends on
future conditions as well as current reserves. Therefore, in our model, proactive
recapitalization arises from the interaction among the anticipated
financing-cost path, operating risk, and the value assigned to liquidity at
maturity.

Allowing proactive capital injection changes the geometry of the stochastic control problem. When recapitalization is optimally postponed until reserves reach zero, the lower intervention boundary is fixed by the solvency constraint, and the nontrivial marginal problem is effectively one-sided. When proactive recapitalization is optimal, a positive lower financing boundary must be determined jointly with the upper dividend boundary. The continuation region is the band between these boundaries. Reserves below the lower boundary are increased through capital injection, reserves above the upper boundary are reduced through dividend payments, and reserves inside the band are left uncontrolled. Under the verification conditions established in this paper, the two thresholds become the reflection boundaries of the optimal
policy in the original dividend and capital injection control problem.

To make this structure precise, let $V(t,x)$ denote the optimal shareholder
value at time $t$ when the current reserve is $x$, and let $f(t)$ and $m(t)$
denote, respectively, the current marginal reward from dividends and the
current marginal cost of capital injection. The possibility of making an
immediate dividend payment gives a lower bound on the marginal value of
reserves, whereas the possibility of making an immediate capital injection
gives an upper bound. More precisely, for $h>0$,
\[
    f(t)
    \leq
    \frac{V(t,x+h)-V(t,x)}{h}
    \leq
    m(t).
\]
Thus, wherever the reserve derivative exists, one has  $f(t)\leq V_x(t,x)\leq m(t)$. The lower gradient constraint becomes binding in the dividend region,
where an additional unit of reserves is immediately distributed, while the
upper gradient constraint becomes binding in the financing region, where
an additional unit of reserves can be acquired at the marginal cost
$m(t)$. Between these regions, the marginal value evolves according to a
linear parabolic equation. The interaction of the two marginal intervention
conditions therefore produces a double-obstacle problem and, probabilistically,
a Dynkin game.

This game representation is not merely a technical reformulation. It
captures the strategic timing embedded in the marginal value of liquidity.
One stopping player represents the option to realize the marginal dividend
reward, while the other represents the option to acquire an additional
unit of reserve at the marginal financing cost. Under a passive capital injection scenario,
the financing intervention is tied to absorption at the solvency boundary,
and the marginal characterization reduces to a one-player stopping problem.
Under proactive capital injection, the timing of the financing intervention becomes
endogenous, so the appropriate marginal object is a two-player stopping game, known as a Dynkin game. The passage from optimal stopping to a Dynkin game is therefore the
mathematical counterpart of moving from reactive recapitalization to an
independently timed financing decision.

Connections between bounded-variation control and Dynkin games have been
developed in several general and application-specific settings. For example,
\citet{Boetius2005} establishes such a connection for a broad class of
stochastic control problems; \citet{DeAngelisFerrari2014} use a
finite-horizon Dynkin game to characterize two investment boundaries and an
associated reflected investment--disinvestment policy;
\citet{Yang2014} studies time-inhomogeneous finite-horizon singular control
through spatial integration of a game value, including the treatment of
terminal interventions. Our contribution is to implement and extend this
approach for a solvency-constrained dividend problem with proactive capital
injection. The solvency boundary, maturity, and the two endogenous
intervention mechanisms play distinct roles in the marginal game. Moreover,
recovering the original control value from its marginal representation
requires identifying an additive time-dependent term and verifying a
two-sided reflected policy under regularity that is not imposed directly on
the unknown control value.

Our analysis proceeds from the marginal Dynkin game in Section \ref{sec:game} to the original singular control
problem. We first characterize the intervention regions and their dependence on financing conditions.
We then recover the shareholder value and verify the optimality of the corresponding reflected policy.

The main contributions are as follows.

Firstly, we formulate the marginal value of liquidity as a finite-horizon Dynkin game with exogenous termination upon absorption at the solvency boundary or at maturity. We establish the existence of
the game value and saddle-point stopping rules, and show that
spatial monotonicity yields two interval contact regions.
 These
results identify an endogenous lower financing boundary and an upper
dividend boundary for the original dividend and capital injection problem. Relative to the passive capital-injection benchmark of \citet{ferrari2019optimal}, financing in our model can occur at strictly positive reserves, and its timing is determined
jointly with dividend payments.
Table~\ref{tab:FS-comparison} summarizes the main differences.
\begin{table}[H]

		\centering
		\small
		\caption{Comparison with the finite-horizon passive-injection benchmark.}
		\label{tab:FS-comparison}
		\begin{tabular}{p{0.2\textwidth} p{0.34\textwidth}p{0.34\textwidth}}
			\toprule
			Feature & \citet{ferrari2019optimal} & Present paper\\
			\midrule
			Injection choice & Minimal regulator at zero & Independently chosen singular control\\
			Marginal problem & Optimal stopping with absorption & Dynkin game with two obstacles\\
			Financing boundary & Fixed at zero & Endogenous $a(t)\ge0$\\
			Dividend boundary & One moving boundary & Determined jointly with $a(t)$\\
			Verification & Spatial integration of a stopping value & Spatial integration of a game value\\
			Terminal payoff in the verified class & Linear liquidation & Structured concave continuation payoff\\
			\bottomrule
		\end{tabular}

\end{table}

Secondly, our analysis identifies two economic implications.  The financing-cost path governs both the incentive to retain
liquidity and the timing of capital injections. At interior
times, a positive financing boundary requires $m'(t)\ge0$, while a finite dividend boundary requires $f'(t)\le0$. However, the financing sign condition is necessary
but is not sufficient on its own to ensure proactive capital injection (see Figure \ref{fig:beyond-canonical} in Section \ref{sec:numerical}). In addition, holding dividend rewards, the terminal payoff, and the remaining model inputs fixed, we establish comparative statics showing that a higher financing-cost path weakly reduces total shareholder value while weakly increasing the marginal value of liquidity under our verification conditions.	Retained reserves provide protection against more expensive future financing, encouraging liquidity retention. When current financing costs are unchanged, higher future
costs also weakly raise the current financing boundary,	strengthening the incentive to prefund capital needs. These results distinguish the incentive to retain liquidity 	against more expensive future financing from the incentive
to prefund capital needs.

Finally, under a regular intervention-band condition, we connect these marginal Dynkin game results to the original dividend and capital injection control problem by recovering the full shareholder value (i.e., the value function of the control problem) through
spatial integration of the game value,  including the required time-dependent additive term, and verify the optimal dividend and capital injection policy as
two-sided Skorokhod reflection between the identified financing and dividend boundaries.
For a concrete class of primitives with decreasing dividend
rewards, increasing financing costs, and a structured concave
terminal payoff, we establish the Sobolev regularity required
for verification and prove that both boundaries are finite,
continuous, and strictly separated. We also identify their
distinct terminal limits. These results establish the optimal
reflected policy for this class from explicit assumptions on
the model primitives. The terminal payoff can represent a
continuation value assigned to liquidity at the end of the
planning horizon.

The remainder of the paper is organized as follows.
Section~\ref{sec:control} formulates the dividend and capital-injection
problem, interprets the reward and financing coefficients, and establishes
terminal consistency and finiteness.
Section~\ref{sec:game} introduces and analyzes the marginal Dynkin game.
Section~\ref{sec:value-characterization} recovers the control value from the
game and verifies the optimal reflected policy.
Section~\ref{sec:canonical} establishes the regular intervention-band
condition for the canonical class, with the technical boundary arguments
deferred to Appendix~\ref{app:canonical-verification}.
Section~\ref{sec:numerical} presents the numerical method and illustrations,
and Section~\ref{sec:conclusion} concludes the paper.

\section{The Dividend and Capital Injection Problem}\label{sec:control}

Fix a finite horizon $T\in(0,\infty)$. Let $(\Omega,\F,\mathbb F,\Pp)$ carry a standard Brownian motion $W$, where $\mathbb F$ denotes its augmented natural filtration, which satisfies the usual conditions. For an initial calendar time $t\in[0,T]$ and initial surplus $x\ge0$, write $H=T-t$. The uncontrolled surplus is
\begin{equation}\label{eq:uncontrolled}
 A_s(x)=x+\mu s+\sigma W_s,\qquad s\ge0,
\end{equation}
where $\mu\in\R$ and $\sigma>0$ are constants. A dividend process $D$ and a capital injection process $I$ produce the controlled surplus
\begin{equation}\label{eq:controlled}
 X_s^{D,I}=A_s(x)-D_s+I_s,\qquad 0\le s\le H.
\end{equation}
The law of $A(x)$ does not depend on the initial calendar time $t$; $t$ determines the remaining horizon $H=T-t$ and the reward and cost coefficients. We use elapsed time for the controls and calendar time for these coefficients.

\begin{definition}[Admissible controls]\label{def:admissible}
For $(t,x)\in[0,T]\times[0,\infty)$, the admissible set $\A(t,x)$ consists of pairs $(D,I)$ of nondecreasing, finite-valued, $\mathbb F$-adapted c\`adl\`ag processes on $[0,H]$ such that $D_{0-}=I_{0-}=0$,
\[
 X_s^{D,I}\ge0\quad\hbox{for all }s\in[0,H]\ \hbox{a.s.},
 \qquad \E[D_H+I_H]<\infty.
\]
Initial and terminal jumps are allowed.
\end{definition}

Let $f,m:[0,T]\to(0,\infty)$ denote the marginal dividend reward and the marginal cost of capital injection, respectively, and let $g(T,\cdot)$ be the terminal payoff. Define the performance functional for any $(D,I)\in\A(t,x)$ as
\begin{align}\label{eq:payoff}
 J(t,x;D,I):=\E\left[\int_{[0,H]} f(t+s)\dd D_s
            -\int_{[0,H]}m(t+s)\dd I_s+g(T,X_H^{D,I})\right].
\end{align}
Then, the value function is
\begin{equation}\label{eq:value}
    V(t,x):=\sup_{(D,I)\in\A(t,x)}J(t,x;D,I).
\end{equation}
 Although we retain the traditional terminology ``surplus'' from the optimal dividend literature, $X$ can be interpreted as the firm's liquid reserve. It represents funds available for operations and shareholder distributions. The cumulative dividend process $D$ reduces this reserve, whereas $I$ records additions through external financing, for example, equity issuance. Thus, a dividend payment of one unit removes one unit of liquidity, and a capital injection of one unit adds one unit of liquidity. The solvency constraint is imposed throughout the planning horizon, so shareholders cannot choose default as an alternative to meeting a financing requirement.

The functions $f$ and $m$ attach values to these two liquidity transfers. Specifically, $f(t)$ is the reward from distributing one unit of reserve at calendar time $t$, and $m(t)$ is the cost of obtaining one additional unit through external finance. The process $I$ measures the liquidity received by the firm, rather than the gross proceeds from
securities issuance. Consequently, $m(t)\dd I$ is the cost entering the shareholder objective, and $m$ may include proportional issuance fees, offer discounts, and other financing frictions. Such costs are distinct from the one-unit increase in the liquid reserve. The corporate finance evidence discussed in the previous section motivates this distinction between internal liquidity and the cost of externally replenishing it.

All rewards and costs in \eqref{eq:payoff} are expressed in common valuation units. For example, with a constant discount rate $r$ and nominal per-unit coefficients $\eta$ and $\kappa$, we may write
\[
 f(t)=e^{-rt}\eta(t),\qquad m(t)=e^{-rt}\kappa(t).
\]
The time profile of the discounted financing cost then satisfies $m'(t)=e^{-rt}\bigl(\kappa'(t)-r\kappa(t)\bigr)$.
An increase in the nominal coefficient $\kappa$ is therefore different from an increase in $m$, where the latter occurs when $\kappa'(t)>r\kappa(t)$. A constant nominal cost gives a decreasing discounted cost when $r>0$, but discounting does not force $m$ to decrease if nominal financing conditions deteriorate sufficiently quickly. The intervention-timing restrictions below concern $m$ itself, because this coefficient determines the cost assigned to capital injections at different dates in the objective.

In the classical passive capital injection structure, capital is supplied only when needed to keep reserves nonnegative. The injection process is the minimal regulator of the dividend-controlled surplus at zero, as in the solvency-preserving formulation of \citet{KS08}. In \citet{ferrari2019optimal}, minimal injection at zero is part of the model specification. Their Remark~2.4 identifies decreasing $m$, together with $\min_{s\in[0,T]}m(s)>g_x(T,x)$ for every $x\ge0$, as sufficient conditions for this rule to remain optimal when capital injection is added to the problem. The condition on the terminal marginal value matters because the value of carrying additional funds to maturity also affects whether early capital injection is profitable. When the lower boundary is fixed at zero, the marginal control problem can be formulated as an optimal stopping problem for an absorbed diffusion.

In our model, both $D$ and $I$ are chosen as singular controls, and $m$ may increase over part of the horizon. If financing is expected to become more expensive, waiting until liquidity is exhausted need no longer be optimal. Shareholders may instead supply capital while reserves are positive, securing funds under the currently more favorable conditions. Then, the lower financing threshold must be determined jointly with the upper dividend threshold. This possibility motivates the two stopping decisions in the associated Dynkin game. The presence of a positive financing region also depends on operating risk, future dividend opportunities, and the terminal value of liquidity. Our results distinguish a necessary sign condition on $m'$ from the additional conditions that generate a positive financing region.

As in \citet{ferrari2019optimal}, the marginal rewards and costs are deterministic functions of calendar time and do not depend on the current level of reserve. They provide a reduced-form description of the anticipated financing conditions while preserving a one-dimensional stochastic state. Allowing the reward and cost functions to depend on reserves would lead to state-dependent obstacles in the Dynkin game, while stochastic financing conditions would introduce additional state variables. These are natural extensions of the problem, but they lie outside the scope of the present
model.

Finally, $g(T,x)$ is the terminal value when the firm holds $x$ units of liquid reserves at the end of the planning horizon $T$. It may represent a liquidation value or an exogenously specified continuation value of the firm. The latter interpretation is relevant when retained liquidity has value for funding opportunities beyond the planning horizon. Unlike $f$ and $m$, the terminal payoff $g$ depends on the reserve. We do not assume that it is simply the cash distribution $f(T)x$, but permit a general concave terminal value with decreasing marginal value of liquidity, and later introduce a canonical specification of the terminal payoff.

To proceed, we introduce the following conditions, which ensure that the specified terminal value is consistent with the interventions available at maturity.

\begin{assumption}\label{assumption2_1}
The functions $f,m$ belong to $C^1([0,T])$ and are strictly positive. The function $g(T,\cdot)\in C^1([0,\infty))$ is nondecreasing and concave, $g(T,0)$ is finite, and $g_x(T,\cdot)$ is bounded and Lipschitz continuous. Moreover,
\begin{enumerate}[label=(\roman*)]
\item\label{ass:spread} $m(t)>f(t)$ for $0\le t\le T$;
\item\label{ass:terminal} $f(T)\le g_x(T,x)\le m(T)$ for $x\ge0$.
\end{enumerate}
\end{assumption}
Condition~\ref{ass:spread} requires the cost of raising one unit of liquidity to exceed the reward from immediately distributing it. It therefore excludes a profitable same-date financing--dividend operation. Condition~\ref{ass:terminal} ensures that the terminal value is consistent with a final reserve adjustment. If $g_x(T,x)<f(T)$, a marginal unit is worth more as a dividend than when retained in the terminal payoff. If $g_x(T,x)>m(T)$, an additional unit supplied through financing increases terminal value by more than it costs. The two inequalities rule out both types of profitable marginal adjustment at maturity. The next proposition gives the corresponding statement for arbitrary, possibly large terminal interventions.
The Lipschitz continuity of $g_x(T,\cdot)$ is used in the terminal estimates for the marginal game, see Proposition~\ref{prop:viscosity-characterization} and Lemma~\ref{lem:app-regularity}. In the canonical specification in Section \ref{sec:canonical}, the additional compatibility $g_x(T,0)=m(T)$ makes the estimate uniform up to the solvency boundary. The Lipschitz continuity of the terminal payoff $g(T,\cdot)$ itself follows from boundedness of $g_x(T,\cdot)$.

\begin{proposition}\label{prop:terminal-envelope}
For a primitive terminal payoff $g_0$, optimization over final dividends and capital injections gives
\begin{equation}\label{eq:terminal-envelope}
 \widehat g(x)=\sup_{y\ge0}\{g_0(y)+f(T)(x-y)^+-m(T)(y-x)^+\}.
\end{equation}
 For a continuously differentiable payoff $g(T,\cdot)$, applying \eqref{eq:terminal-envelope} with $g_0=g(T,\cdot)$ leaves $g$ unchanged if and only if
$f(T)\le g_x(T,x)\le m(T)$ for every $x\ge0$, using a right derivative at zero. Consequently, under Assumption~\ref{assumption2_1}, $V(T,x)=g(T,x)$ and terminal jumps may be omitted without loss of generality.
\end{proposition}
\begin{proof}
For a net terminal reserve $y=x-d+i$, simultaneous dividends and injections reduce the payoff by $(m(T)-f(T))\min(d,i)$. Removing this common amount yields \eqref{eq:terminal-envelope}. If the derivative bounds hold, integration gives
\[
 f(T)d-m(T)i+g(T,x-d+i)\le g(T,x)
\]
for every feasible $d,i\ge0$, with equality for $d=i=0$. Conversely, the absence of a profitable marginal capital injection gives $g_x\le m(T)$, and the absence of a profitable marginal dividend gives $g_x\ge f(T)$ at $x>0$. Continuity extends the latter bound to zero. At $T$, these are the only available interventions.
\end{proof}

Thus Assumption~\ref{assumption2_1}(ii) requires the chosen terminal value to incorporate the firm's intervention opportunities at maturity.  If one begins with a primitive payoff $g_0$ that violates the marginal bounds, the relevant terminal value is its intervention envelope as given in \eqref{eq:terminal-envelope}, rather than $g_0$ itself. Under the stated bounds, terminal jumps remain admissible, but cannot improve the objective; hence, the control value satisfies $V(T,x)=g(T,x)$.

Note that Assumption~\ref{assumption2_1}\ref{ass:spread} alone does not exclude profitable operations across different dates. This issue requires attention because capital injection is an independent and unbounded control, rather than the regulator needed to prevent insolvency. For example, if $f(s)>m(r)$ at dates $r<s$, an extra injection of $K$ units at $r$, followed by an extra dividend of $K$ units at $s$, adds $K(f(s)-m(r))$ to the payoff. The additional reserves remain nonnegative over this period, so these operations can be superimposed on an admissible policy. Their gain tends to infinity as $K$ increases.

On the other hand, concavity and boundedness of the terminal marginal payoff give the limit $ g_x^\infty:=\lim_{x\to\infty}g_x(T,x)$.
This is the marginal terminal value of liquidity when the firm already holds a sufficiently large reserve. If $g_x^\infty>m(r)$, arbitrarily large additional capital injections at $r$, carried to maturity, likewise generate unbounded gains. We therefore impose the following sufficient condition for finiteness, covering both future dividend distributions and terminal accumulation.

\begin{assumption}\label{assumption2_2}
For every $t\in[0,T]$,
\begin{equation}\label{eq:cross-time-condition}
 m(t)>\max\left\{\sup_{t\le s\le T}f(s),\ g_x^\infty\right\}.
\end{equation}
\end{assumption}

Assumption~\ref{assumption2_2} requires the financing cost at each date to exceed the largest marginal rate at which an arbitrarily large additional reserve can subsequently be converted into value. Its terminal component uses the asymptotic slope $g_x^\infty$, rather than $\sup_{x\ge0}g_x(T,x)$. This distinction preserves an economically relevant incentive for a proactive capital injection. Because $g(T,\cdot)$ is concave, its marginal value decreases with reserves; it may exceed $m(t)$ at low reserves while falling below $m(t)$ for sufficiently large reserves. Additional capital can then be valuable up to a finite amount without making unlimited accumulation profitable. In particular, the well-posedness condition is compatible with the proactive financing mechanism studied below. The strict inequality in \eqref{eq:cross-time-condition} is a convenient sufficient condition for finiteness.

\begin{lemma}\label{lem:finiteness}
Under Assumptions~\ref{assumption2_1} and \ref{assumption2_2},
the value function $V(t,x)$ given by \eqref{eq:value} is finite for every
$(t,x)\in[0,T]\times[0,\infty)$.
\end{lemma}

\begin{proof}
Define
\[
q_0(t):=\max\left\{\sup_{t\le s\le T}f(s),\,g_x^\infty\right\}.
\]
Assumption~\ref{assumption2_2} and continuity imply that $\delta:=\min_{0\le t\le T}\{m(t)-q_0(t)\}>0.$
Set
\[
q(t)
:=
\max\left\{
\sup_{t\le s\le T}f(s),\,
g_x^\infty+\frac{\delta}{2}
\right\}.
\]
Then $q$ is bounded and nonincreasing, and we have $f(t)\le q(t)$ and $q(t)\le m(t)-\delta/2$.
Since $g_x(T,x)\to g_x^\infty$ as $x\to\infty$, there exists a constant
$C<\infty$ such that $g(T,x)\le C+q(T)x$ for  $x\ge0$. For any $(D,I)\in\A(t,x)$, integration by parts applied to
$q(t+s)X_s^{D,I}$ gives
\begin{align*}
J(t,x;D,I)
\le{}&
C+q(t)x+\mu\E\Bigg[\int_0^H q(t+s)\,\dd s\Bigg] +\E\Bigg[\int_{[0,H]}
\bigl(f(t+s)-q(t+s)\bigr)\,\dd D_s\Bigg]\\
&+\E\Bigg[\int_{[0,H]}
\bigl(q(t+s)-m(t+s)\bigr)\,\dd I_s\Bigg] +\E\Bigg[\int_{(0,H]} X_{s-}^{D,I}\,\dd q(t+s)\Bigg],
\end{align*}
where the stochastic integral with respect to $W$ has zero expectation and has been omitted.  The dividend and capital injection terms are nonpositive, and the last term is nonpositive since $X^{D,I}\ge0$ and $q$ is nonincreasing. In addition,  since $q$ is nonincreasing, we also have
\[
\int_0^H q(t+s)\,\dd s\le Hq(t).
\]
Therefore
$ J(t,x;D,I)\le C+q(t)x+|\mu|Hq(t)$
uniformly over $(D,I)\in\A(t,x)$. Writing $Q=\sup_{0\le s\le T}q(s)<\infty$, we obtain $V(t,x)\le C+Qx+|\mu|TQ$.
Thus the upper bound is finite for each reserve level and grows at most linearly in $x$, uniformly in $t$.

For a finite lower bound, we simply take $D\equiv0$ and let $I^0$ be the minimal
Skorokhod regulator at zero.  Since $m$ is bounded, the regulator has finite expectation over a finite horizon, and $g(T,\cdot)\ge g(T,0)>-\infty$, this policy yields a finite lower bound. Together with the upper bound established above, this proves that $V(t,x)$ is finite.
\end{proof}

\section{The Marginal Dynkin Game}\label{sec:game}

The auxiliary game introduced below extends the finite-horizon stopping
connection of \citet{ferrari2019optimal}.
In their setting, capital injection
is the minimal regulator required to keep the surplus nonnegative, so the
dividend decision is the only singular control, and the marginal
value is characterized by a one-player optimal stopping problem.
 In the present model, the injection process is chosen independently alongside dividends.
Hence, the marginal value must account for two possible interventions:
paying an additional unit of dividend yields the marginal reward $f(t)$,
whereas supplying an additional unit of capital incurs the marginal cost
$m(t)$.  This naturally leads to a stopping problem with two obstacles,
$f$ and $m$, and hence to a zero-sum Dynkin game.
We formulate this auxiliary Dynkin game below and show in
Section~\ref{sec:value-characterization} that its value determines the
marginal value of the original dividend and  capital injection control problem.
Throughout the paper, we use $V$ to denote the control value of the problem \eqref{eq:value} and use $v$ to denote the game value; Theorem~\ref{thm:theorem-main} establishes $V_x=v$ under the verification conditions.

Fix $(t,x)\in[0,T]\times[0,\infty)$, write $H=T-t$, and recall the uncontrolled process $A_s(x)$ from \eqref{eq:uncontrolled}. Define the hitting time
\begin{equation}\label{eq:hitting}
 S(x):=\inf\{s\ge0:A_s(x)\le0\}.
\end{equation}
Set $\rho=H\wedge S(x)$.
Let $\Tset_\rho$ be the set of stopping times with values in $[0,\rho]$.
The dividend player chooses $\tau\in\Tset_\rho$ to maximize the payoff,
and the capital-injection player chooses $\delta\in\Tset_\rho$ to minimize
it.
If neither player stops before $\rho$, the game terminates exogenously, either when the reserve reaches zero or at maturity.  When the reserve reaches zero before maturity, the terminal
payment of the marginal game is the marginal capital-injection cost
$m(t+S(x))$.  If maturity is reached before the solvency boundary, the terminal payment is the terminal marginal value
$g_x(T,A_H(x))$.  Accordingly, we define the terminal payoff by
\begin{equation}\label{eq:Z}
 Z_{t,x}:=m(t+\rho)\ind_{\{S(x)<H\}}
       +g_x(T,A_H(x)^+)\ind_{\{H\le S(x)\}}.
\end{equation}
For $t<T$ and $x>0$, the game payoff is
\begin{equation}\label{eq:game-payoff}
 \Gamma_{t,x}(\tau,\delta)=
 \begin{cases}
 f(t+\tau),&\tau\le\delta,\quad\tau<\rho,\\
 m(t+\delta),&\delta<\tau,\quad\delta<\rho,\\
 Z_{t,x},&\tau\wedge\delta=\rho.
 \end{cases}
\end{equation}
Thus, simultaneous stopping before $\rho$ is assigned the dividend payoff.
The lower and upper values of the game are defined, respectively, as
\begin{equation}\label{eq:game-values}
 \underline v(t,x)=\sup_{\tau\in\Tset_\rho}\inf_{\delta\in\Tset_\rho}
       \E[\Gamma_{t,x}(\tau,\delta)],\qquad
 \overline v(t,x)=\inf_{\delta\in\Tset_\rho}\sup_{\tau\in\Tset_\rho}
       \E[\Gamma_{t,x}(\tau,\delta)].
\end{equation}

\begin{proposition}\label{prop:game-existence}
Under Assumption~\ref{assumption2_1}, the game has a bounded value
$v=\underline v=\overline v$. It satisfies
\begin{equation}\label{eq:obstacle-bounds}
 f(t)\le v(t,x)\le m(t),\qquad t<T,\ x>0.
\end{equation}

For $t<T$ and $x>0$, a saddle point is given by the first contact times
\begin{align*}
 \tau_v&=\inf\{s\in[0,\rho):v(t+s,A_s(x))=f(t+s)\}\wedge\rho,\\
 \delta_v&=\inf\{s\in[0,\rho):v(t+s,A_s(x))=m(t+s)\}\wedge\rho.
\end{align*}
\end{proposition}
\begin{proof}
To reformulate the random-horizon stopping game as a standard Dynkin game on the deterministic interval $[0, H]$, define for $0\le s\le H$,
\begin{align*}
    L_s &:=f(t+s)\ind_{\{s<\rho\}}+Z_{t,x}\ind_{\{s\ge\rho\}},\\
    U_s & :=m(t+s)\ind_{\{s<\rho\}}+Z_{t,x}\ind_{\{s\ge\rho\}}.
\end{align*}
Then, the processes $L$ and $U$ are $\mathbb F$-adapted and c\`adl\`ag.  Indeed, before
$\rho$, they are given by the continuous deterministic functions
$f(t+\cdot)$ and $m(t+\cdot)$, respectively; at $\rho$ they may have
one jump to $Z_{t,x}$, and they remain constant thereafter.  Moreover,
Assumption~\ref{assumption2_1}\ref{ass:spread} gives $L_s\le U_s$, for $0\le s\le H$. Since $f$, $m$, and $g_x(T,\cdot)$ are nonnegative and bounded, $L$ and $U$ are uniformly bounded, so the integrability requirement in the Dynkin-game existence theorem (see e.g., \cite{dumitrescu2016dynkin}) is automatically satisfied.
We verify Mokobodzki\textquotesingle s condition in the form of \citet[Definition 3.9]{dumitrescu2016dynkin}.  Define
\[
Y_s^0
:=
\frac{f(t+s)+m(t+s)}{2}\ind_{\{s<\rho\}}
+
Z_{t,x}\ind_{\{s\ge\rho\}}.
\]
Then, $L_s\le Y_s^0\le U_s$, for $0\le s\le H$. Moreover, $Y^0$ is a bounded, $\mathbb F$-adapted, and finite-variation process, i.e., before
$\rho$ it follows the continuously differentiable function $(f+m)/2$,
it has at most one jump at $\rho$, and it is constant thereafter. Its total variation is bounded by a deterministic constant, because the preterminal path is continuously differentiable and its only possible jump is bounded. Write
$Y_s^0=Y_0^0+K_s^+-K_s^-$ for its Jordan decomposition, with $K_0^\pm=0$.
Then, we define two auxiliary processes
\[
 P_s:=\E[Y_H^0+K_H^-\mid\F_s]-K_s^-,\qquad
 Q_s:=\E[K_H^+\mid\F_s]-K_s^+,
\]
which are bounded nonnegative supermartingales, and $P_s-Q_s=Y_s^0$.
Since $L\le Y^0\le U$, Mokobodzki's condition follows.
 Then, for  $\mathbb F$-stopping times $\tau$ and $\delta$ taking values in $[0,H]$, the payoff $L_\tau\ind_{\{\tau\le\delta\}}+U_\delta\ind_{\{\delta<\tau\}}$
coincides with \eqref{eq:game-payoff} after replacing $\tau$ and $\delta$
by $\tau\wedge\rho$ and $\delta\wedge\rho$.  Therefore, the standard
results of \citet[Proposition 3.10 and Theorem 3.5]{dumitrescu2016dynkin} yield
\[
\sup_{\tau\in\Tset_\rho}
\inf_{\delta\in\Tset_\rho}
\E\!\left[\Gamma_{t,x}(\tau,\delta)\right]
=
\inf_{\delta\in\Tset_\rho}
\sup_{\tau\in\Tset_\rho}
\E\!\left[\Gamma_{t,x}(\tau,\delta)\right].
\]
Thus the game has a value $v(t,x)$.

At $\rho$, the lower obstacle $L$ can only jump upwards and the upper obstacle $U$ can only jump downwards: the absorption payoff is $m(t+\rho)$, while the maturity payoff lies between $f(T)$ and $m(T)$. Both obstacles are continuous elsewhere. Hence $L$ and $-U$ are left-upper-semicontinuous along stopping times. By \citet[Theorem 3.7]{dumitrescu2016dynkin}, their first contact times with the value process form a saddle point. The strong Markov property identifies the value process before $\rho$ with $v(t+s,A_s(x))$, giving the stopping times stated above.

Suppose now that $t<T$ and $x>0$, so that $\rho>0$ almost surely.
If the maximizing (dividend) player chooses $\tau=0$, then, irrespective of
$\delta$, the payoff is at least $f(t)$.  Hence, $v(t,x)\ge f(t)$. Similarly, if the minimizing (capital injection) player chooses $\delta=0$, then the payoff is at most $m(t)$, and therefore
$v(t,x)\le m(t)$.
This proves \eqref{eq:obstacle-bounds}.
\end{proof}

\subsection{Viscosity characterization and comparison}
\label{subsec:viscosity-comparison}

The Dynkin-game value can be characterized as a viscosity solution of
a double-obstacle problem, using the viscosity solution framework of \citet{crandall1992users}. Let
\[
 \mathcal Q=([0,T]\times[0,\infty))\setminus\{(T,0)\},\qquad
 \mathcal Q^\circ=(0,T)\times(0,\infty),
\]
and define the operator
\[
 \mathcal L:=\partial_t+\mu\partial_x+\frac{\sigma^2}{2}\partial_{xx}.
\]
We exclude $(T,0)$ because the boundary condition at the solvency level
and the terminal condition need not be compatible there: approaching along
$x=0$ gives $m(T)$, whereas approaching along $t=T$ gives
$g_x(T,0)$.

For a locally bounded function $z$, let $z^*$ and $z_*$ denote its upper
and lower semicontinuous envelopes, respectively. At boundary points, these
envelopes are understood through limits taken from the interior
$\mathcal Q^\circ$.
\begin{proposition}
\label{prop:viscosity-characterization}
Under Assumption~\ref{assumption2_1}, the game value $v$ is the bounded viscosity solution of
\begin{equation}\label{eq:double-obstacle}
    \max\left\{
        v-m,\,
        \min\left[v-f,-\mathcal Lv\right]
    \right\}=0
    \qquad\text{on }\mathcal Q^\circ,
\end{equation}
in the sense that $v^*$ is a viscosity subsolution and $v_*$ is a viscosity
supersolution. Moreover,
\begin{align}
    \lim_{\substack{(s,y)\to(t,0)\\(s,y)\in\mathcal Q^\circ}}
        v(s,y)
    &=m(t),
    &&0\leq t<T,
    \label{eq:viscosity-boundary-limit-zero}\\
    \lim_{\substack{(s,y)\to(T,x)\\(s,y)\in\mathcal Q^\circ}}
        v(s,y)
    &=g_x(T,x),
    &&x>0.
    \label{eq:viscosity-boundary-limit-terminal}
\end{align}
The convergence in \eqref{eq:viscosity-boundary-limit-zero} is  uniform for $t$ in compact subsets of $[0,T)$, and  the convergence in \eqref{eq:viscosity-boundary-limit-terminal} is uniform for $x$ in compact subsets of $(0,\infty)$.
\end{proposition}

\begin{proof}
The strong Markov property, together with concatenation of conditional
$\varepsilon$-optimal stopping rules, yields the dynamic programming
principle up to the first exit time from any relatively compact
space--time cylinder contained in $\mathcal Q^\circ$. The usual
test-function argument then shows that $v^*$ is a viscosity subsolution
and $v_*$ is a viscosity supersolution of
\eqref{eq:double-obstacle}.

It remains to verify the boundary limits. First, fix a compact interval
$K\subset[0,T)$. The absorption time $S(x)$ converges to zero in probability
as $x\downarrow0$, uniformly over starting times $t\in K$. The minimizing
player can stop immediately, giving $v(t,x)\leq m(t)$.  For the reverse estimate, let the maximizing player wait until absorption
or maturity. If the minimizing player stops first at stopping time $\delta$, the payoff is
$m(t+\delta)$; if absorption occurs first, the payoff is
$m(t+S(x))$. The only other possibility is that the state remains positive
until maturity, whose probability converges to zero as $x\downarrow0$,
uniformly for $t\in K$. Boundedness and uniform continuity of the payoff
functions therefore imply
\[
    v(t,x)\longrightarrow m(t)
    \qquad\text{as }x\downarrow0,
\]
locally uniformly for $t<T$. Next, let $K$ be a compact subset of $(0,\infty)$. As $t\uparrow T$, the
probability of absorption before maturity converges to zero uniformly for
$x\in K$, and $A_{T-t}(x)\longrightarrow x$ in probability. If the minimizing player waits until absorption or
maturity, the resulting upper estimate converges to $g_x(T,x)$ because
$f(T)\leq g_x(T,x)$. If the maximizing player waits instead, the
corresponding lower estimate converges to $g_x(T,x)$ because
$g_x(T,x)\leq m(T)$. The Lipschitz continuity of
$g_x(T,\cdot)$ makes this convergence locally uniform in $x$. This proves
\eqref{eq:viscosity-boundary-limit-zero} and
\eqref{eq:viscosity-boundary-limit-terminal}.
\end{proof}

The preceding limits in \eqref{eq:viscosity-boundary-limit-zero} and
	\eqref{eq:viscosity-boundary-limit-terminal} allow us to extend the interior game value by
	\begin{equation}\label{eq: game-bound}
			v(t,0):=m(t),\qquad 0\le t<T,
		\qquad
		v(T,x):=g_x(T,x),\qquad x>0.
	\end{equation}
	At the terminal corner $(T,0)$, we set $v(T,0):=g_x(T,0)$
	in accordance with the maturity convention.

\begin{remark}
\label{rem:terminal-corner}
 Continuity at the corner $(T,0)$ is not assumed unless we require $g_x(T,0)=m(T)$. To be specific, take $\mu=0$, $f\equiv1$, $m\equiv2$, and $g(T,x)=x$.
Both players waiting until $\rho$ is a saddle point: an earlier dividend stop pays $1$, while an earlier financing stop pays $2$.
Consequently, for $t<T$ and $x>0$,
\[
 v(t,x)=1+\Pp(S(x)<T-t)
 =1+2\left[1-\Phi\left(\frac{x}{\sigma\sqrt{T-t}}\right)\right],
\]
where $\Phi$ is the standard normal distribution function.
This value is continuous on $\mathcal Q$, but its boundary limits at $(T,0)$ differ: the solvency boundary gives $2$, whereas maturity gives $1$. In the canonical specification in Section \ref{sec:canonical}, we let $g_x(T,0)=m(T)$, and Lemma~\ref{lem:app-regularity} establishes continuity also at this corner.
\end{remark}

We next establish a comparison principle for the double-obstacle problem.

\begin{lemma}[Comparison principle]
\label{lem:halfline-comparison}
Let $u$ be a bounded upper semicontinuous viscosity subsolution and let $w$
be a bounded lower semicontinuous viscosity supersolution of
\eqref{eq:double-obstacle} on $\mathcal Q^\circ$. Suppose that
\begin{align*}
    u^*(t,0)&\leq w_*(t,0),
    &&0\leq t<T,\\
    u^*(T,x)&\leq w_*(T,x),
    &&x\ge0.
\end{align*}
Then
\[
    u\leq w
    \qquad\text{on }~\mathcal Q^\circ.
\]
\end{lemma}

\begin{proof}
Choose $\lambda>0$ sufficiently large that $ \lambda(1+x^2)-2\mu x-\sigma^2>0$, for every $x\geq0$, and define an auxiliary function $\Psi(t,x):=e^{\lambda(T-t)}(1+x^2)$. Then
\[
    -\mathcal L\Psi
    =
    e^{\lambda(T-t)}
    \left[
        \lambda(1+x^2)-2\mu x-\sigma^2
    \right]
    >0.
\]

For $\varepsilon,\eta>0$, set
\[
w^{\varepsilon,\eta}(t,x)=w(t,x)+\varepsilon\Psi(t,x)+\frac{\eta}{t},\qquad 0<t<T.
\]
The additional term satisfies $-\mathcal L(\eta/t)=\eta/t^2>0$. Suppose that a smooth function $\phi$ touches $w^{\varepsilon,\eta}$ from below
at a point where $w^{\varepsilon,\eta}<m$. Then
$\phi-\varepsilon\Psi-\eta/t$ touches $w$ from below at the same point. Since
$w<w^{\varepsilon,\eta}<m$, the supersolution property of $w$ gives
\[
    w-f\geq0,
    \qquad
    -\mathcal L(\phi-\varepsilon\Psi-\eta/t)\geq0.
\]
Therefore,
\[
    w^{\varepsilon,\eta}-f>0,
    \qquad
    -\mathcal L\phi
    =
    -\mathcal L(\phi-\varepsilon\Psi-\eta/t)
    -\varepsilon\mathcal L\Psi+\eta/t^2
    >0.
\]
Thus, $w^{\varepsilon,\eta}$ satisfies the supersolution inequalities strictly whenever it lies below the upper obstacle.
We now argue by contradiction. If  $\sup_{\mathcal Q^\circ}(u-w^{\varepsilon,\eta})>0$, the boundary ordering, the growth of $\Psi$ at spatial infinity, and the term $\eta/t$ near $t=0$ imply that this positive maximum is attained in $\mathcal Q^\circ$. At such a maximum,
\[
    f\leq w<w^{\varepsilon,\eta}<u\leq m.
\]
Hence, $u$ lies strictly above the lower obstacle and $w^{\varepsilon,\eta}$ lies
strictly below the upper obstacle. The standard doubling-of-variables argument (see e.g., \cite{crandall1992users}) applied to the subsolution $u$ and the strict supersolution
$w^{\varepsilon,\eta}$ then yields a contradiction. Therefore, $u\leq w^{\varepsilon,\eta}$ on $\mathcal Q^\circ$. Letting first $\eta\downarrow0$ and then $\varepsilon\downarrow0$ proves $u\leq w$.

\end{proof}

The viscosity characterization and comparison principle now give the two structural
properties needed to define the intervention regions.

\begin{proposition}
\label{prop:intervention-sign}
For $t\in(0,T)$ and $x>0$,
\begin{align}
    v(t,x)=f(t)
    &\quad\Longrightarrow\quad
    f'(t)\leq0,
    \label{eq:dividend-sign}\\
    v(t,x)=m(t)
    &\quad\Longrightarrow\quad
    m'(t)\geq0.
    \label{eq:injection-sign}
\end{align}

\end{proposition}

\begin{proof}
Suppose first that $v(t,x)=m(t)$. Since $v\leq m$, the function
$(s,y)\mapsto m(s)$ touches $v^*$ from above at $(t,x)$. The viscosity
subsolution property gives
\[
    \max\left\{
        0,\,
        \min\left[m(t)-f(t),-m'(t)\right]
    \right\}
    \leq0.
\]
Because $m(t)>f(t)$, it follows that $m'(t)\geq0$.

Similarly, if $v(t,x)=f(t)$, then $v\geq f$, so
$(s,y)\mapsto f(s)$ touches $v_*$ from below at $(t,x)$. The viscosity
supersolution property gives
\[
    \max\left\{
        f(t)-m(t),\,
        \min\left[0,-f'(t)\right]
    \right\}
    \geq0.
\]
Since $f(t)<m(t)$, this implies $f'(t)\leq0$.
\end{proof}

Note that the same implications \eqref{eq:dividend-sign} and \eqref{eq:injection-sign} also hold at $t=0$, as shown by the direct stopping argument in Remark \ref{rem:timing-direct-proof}.
\begin{remark}
\label{rem:timing-direct-proof}
Suppose that $f'(t)>0$. For sufficiently small $\varepsilon>0$, we have
\[
    f(t+\varepsilon)>f(t),
    \qquad
    \min_{0\leq s\leq\varepsilon}m(t+s)>f(t).
\]
Let the maximizing player choose $\tau_\varepsilon=\varepsilon\wedge S(x)$.
If the minimizing player stops the game before $\tau_\varepsilon$, the payoff is $m(t+\delta)>f(t)$. Otherwise, the payoff is either
$f(t+\varepsilon)>f(t)$ or the payoff associated with absorption before
$t+\varepsilon$, which is again strictly greater than $f(t)$. Hence
$v(t,x)>f(t)$. Now suppose that $m'(t)<0$. Choose $\varepsilon>0$ sufficiently small that
$m$ is strictly decreasing on $[t,t+\varepsilon]$ and
$ c_f:=\max_{0\leq s\leq\varepsilon}f(t+s) <m(t)$. Let the minimizing player choose $\delta_\varepsilon=\varepsilon\wedge S(x)$. For every strategy of the maximizing player, the payoff is bounded above
by $Y := \max\left\{ c_f,\, m\bigl(t+\varepsilon\wedge S(x)\bigr)\right\}$. Because $x>0$ implies $S(x)>0$ almost surely and $m$ is strictly decreasing
on $[t,t+\varepsilon]$, we have $Y<m(t)$ almost surely. Since $Y$ is
bounded, $v(t,x)\leq \E[Y]<m(t)$.
Thus $v(t,x)<m(t)$ when $m'(t)<0$. These arguments also apply  at $t=0$.
\end{remark}

\begin{proposition}[Spatial monotonicity]
\label{prop:monotonicity}
Under Assumption~\ref{assumption2_1}, the function
$x\mapsto v(t,x)$ is nonincreasing for every $t\in[0,T]$.
\end{proposition}

\begin{proof}
Write $ G(x)=g_x(T,x)$. Since $g(T,\cdot)$ is concave, $G$ is nonincreasing. The assertion at
$t=T$ therefore follows from $v(T,x)=G(x)$.

Fix $t<T$, $x>0$, and $h>0$. Couple the state processes starting from
$x$ and $x+h$ by using the same Brownian motion. Then $ A_s(x+h)=A_s(x)+h$ up to the absorption time
$ S(x)=\inf\{s\geq0:A_s(x)\leq0\}$. In particular, the process starting from $x+h$ has not been absorbed
before $S(x)$.

Let $\rho_x=(T-t)\wedge S(x)$ be the forced termination time of the game starting from $x$. Consider
an arbitrary stopping time $\tau$ of the maximizing player in the game
starting from $x+h$, and define $\tau_x=\tau\wedge\rho_x$. For each admissible stopping time $\delta$ of the minimizing player in
the game starting from $x$, use the same stopping time in the game
starting from $x+h$. In particular, if
$\delta=\rho_x=S(x)<T-t$, this amounts to voluntary stopping by the
minimizing player when the lower process is absorbed. Before $\rho_x$, the two games have identical stopping payoffs because
$f$ and $m$ do not depend on the state. If the lower process is absorbed
before maturity and neither player has stopped earlier, the game starting from $x$ pays $m(t+S(x))$. The coupled game
starting from $x+h$ pays either this amount or, in the case of simultaneous stopping, $f(t+S(x))\le m(t+S(x))$. If instead $\rho_x=T-t$, their terminal payoffs
satisfy
\[
    G\bigl(A_{T-t}(x)+h\bigr)
    \leq
    G\bigl(A_{T-t}(x)\bigr),
\]
because $G$ is nonincreasing. Thus, under the coupled stopping times, the
payoff in the game starting from $x+h$ is no greater than that in the
game starting from $x$.

It follows that, for every admissible $\tau$ in the game starting from
$x+h$,
\[
    \inf_{\delta}\E[\Gamma_{t,x+h}(\tau,\delta)]
    \leq
    \inf_{\delta}\E[\Gamma_{t,x}(\tau_x,\delta)]
    \leq v(t,x).
\]
Taking the supremum over $\tau$ yields
\[
    v(t,x+h)\leq v(t,x).
\]

Finally, when $x=0$, the boundary condition and the upper-obstacle bound
give
\[
    v(t,h)\leq m(t)=v(t,0),\qquad t<T.
\]
Therefore, $x\mapsto v(t,x)$ is nonincreasing for every $t\in[0,T]$.
\end{proof}

For $t<T$, define the intervention thresholds
\begin{equation}\label{eq:lower-boundary-def}
 a(t):=\sup\{x\ge0:v(t,x)=m(t)\},
\end{equation}
\begin{equation}\label{eq:upper-boundary-def}
 b(t):=\inf\{x>0:v(t,x)=f(t)\},
\end{equation}
with $\inf\varnothing=\infty$. The first set in \eqref{eq:lower-boundary-def} is nonempty because it contains zero. Spatial monotonicity in Proposition \ref{prop:monotonicity} implies that the capital injection set is a lower
interval and the dividend set is an upper interval. Denote the
corresponding subsets of $[0,T)\times(0,\infty)$ by
\begin{align*}
   \mathcal S_I&=\{(t,x)\in[0,T)\times(0,\infty):  v(t,x)=m(t)\},\\
   \mathcal S_D&=\{(t,x)\in[0,T)\times(0,\infty): v(t,x)=f(t)\}, \\
    \mathcal C& =\{(t,x)\in[0,T)\times(0,\infty): f(t)<v(t,x)<m(t)\}.
\end{align*}
They are pairwise disjoint and $\mathcal S_I \cup \mathcal S_D\cup \mathcal C =  [0,T)\times(0,\infty)$.  In particular,
whenever $a(t)<b(t)$, we have
\begin{equation}\label{eq:continuation-core}
 (a(t),b(t))\subseteq\{x>0:f(t)<v(t,x)<m(t)\}.
\end{equation}
In addition, Proposition~\ref{prop:intervention-sign} yields
\[
 m'(t)<0\ \Longrightarrow\ a(t)=0,
 \qquad f'(t)>0\ \Longrightarrow\ b(t)=\infty.
\]
Hence, a nonnegative external financing-cost derivative ($m'(t)\ge 0$) is necessary for proactive capital injection control.

The following proposition compares the control values, game values, and intervention boundaries under two financing-cost functions, with all other model inputs held fixed.
\begin{proposition}\label{prop:cost-comparison}
Consider two models with the same $f,g,\mu,\sigma,T$ and financing costs
$m_1\le m_2$ on $[0,T]$, each satisfying
Assumptions~\ref{assumption2_1}--\ref{assumption2_2}. Denote their control
values, game values, and two intervention boundaries by $V_i,v_i,a_i,b_i$, respectively.
Then
\[
 V_1(t,x)\ge V_2(t,x),\qquad v_1(t,x)\le v_2(t,x),\qquad
 b_1(t)\le b_2(t),
\]
where infinite boundary values are allowed. At every $t<T$ such that
$m_1(t)=m_2(t)$, one also has $a_1(t)\le a_2(t)$.
\end{proposition}

\begin{proof}
For any fixed pair $(t,x)$, the two models have the same class of admissible controls. For every admissible $(D,I)\in\mathcal{A}(t,x)$,
\[
 J_1(t,x;D,I)-J_2(t,x;D,I)
 =\E\!\left[\int_{[0,T-t]}(m_2(t+s)-m_1(t+s))\,\dd I_s\right]\ge0.
\]
Taking the supremum proves that $V_1(t,x)\ge V_2(t,x)$. Every stopping pair has a weakly larger game payoff under $m_2$, including at absorption; therefore $v_1(t,x)\le v_2(t,x)$. The common lower obstacle gives $\{x>0:v_2(t,x)=f(t)\}\subseteq\{x>0:v_1(t,x)=f(t)\}$,
which implies the ordering of $b_i$. If the current capital injection costs
coincide (i.e., $m_1(t)=m_2(t)$), $v_1\le v_2\le m_2=m_1$ gives
$\{x\ge0:v_1(t,x)=m_1(t)\}\subseteq\{x\ge0:v_2(t,x)=m_2(t)\}$,
and hence we have $a_1(t)\le a_2(t)$.
\end{proof}
Proposition \ref{prop:cost-comparison} shows that higher financing costs weakly reduce shareholder value and weakly increase the game value.  Wherever the marginal-value identity $V_x=v$ is verified, an additional unit of reserves is weakly more valuable because it provides protection against higher future financing costs. If today's financing cost is
unchanged, the lower boundary comparison identifies an incentive to raise
capital earlier (higher $a(t)$). The upper boundary comparison identifies an incentive
to retain reserves longer (higher $b(t)$).

\begin{remark}[The passive capital injection benchmark]\label{rem:FS-benchmark}
If $m'(t)<0$ for every $t<T$, the capital-injection region contains no
positive reserve levels. Thus, the minimizing (capital injection) player chooses $\delta\equiv\rho$ and
\[
 v(t,x)=\sup_{\tau\in\Tset_\rho}
 \E\left[f(t+\tau)\ind_{\{\tau<\rho\}}
                +Z_{t,x}\ind_{\{\tau=\rho\}}\right].
\]
This recovers the marginal stopping structure of the passive-injection benchmark in \citet{ferrari2019optimal}. Proposition~\ref{prop:intervention-sign} gives $a(t)=0$ whenever $m'(t)<0$. In this scenario, under the verification conditions in
	Section~\ref{sec:value-characterization}, an optimal policy
	therefore combines capital injection at zero with dividend
	reflection at the upper boundary $b(t)$.
\end{remark}
\subsection{The regular intervention-band condition}

The preceding results characterize the Dynkin game and its two stopping
regions under the primitive assumptions. We now impose the additional
regularity needed to pass from this marginal stopping game back to the
original dividend and capital injection control problem. The two parts of
the assumption below have distinct roles. Condition~\ref{ass:B1} makes the
stopping boundaries the edges of a continuous, separated intervention
band. This permits the reflected control policy to be constructed in
Section~\ref{sec:value-characterization}. Condition~\ref{ass:B2} provides the regularity needed there to
apply the generalized It\^o--Krylov formula and to recover the control value
by spatially integrating the game value. Appendix~\ref{app:canonical-verification} verifies both conditions for the canonical class considered in Section~\ref{sec:canonical}.

\begin{assumption}\label{assumption3_1}
\begin{enumerate}[label=(B\arabic*)]
\item\label{ass:B1}
The boundaries $a,b$ are finite and continuous on $[0,T)$ and have finite
limits $a_T,b_T$ at $T$, with
\[
 0\le a(t)<b(t)<\infty\quad(t<T),\qquad a_T<b_T.
\]
\item\label{ass:B2}
For some $p>3$,
\[
 v\in C(\mathcal Q)\cap
 W^{1,2}_{p,\mathrm{loc}}(\mathcal Q^\circ),
\]
where $W^{1,2}_{p,\mathrm{loc}}$ denotes the standard local parabolic Sobolev space.
\end{enumerate}
\end{assumption}

We record the consequences of this assumption that will be used in Section~\ref{sec:value-characterization}.

\paragraph{The stopping and continuation regions.}
By the definitions of $a$ and $b$, the spatial monotonicity of $v$, and
the continuity in \ref{ass:B2}, the contact sets at every $t<T$ are
\begin{equation}\label{eq:contact-sets}
 \{x\geq0:v(t,x)=m(t)\}=[0,a(t)],
 \qquad
 \{x>0:v(t,x)=f(t)\}=[b(t),\infty).
\end{equation}
Consequently,
\[
 \mathcal C=\{(t,x):t<T,\ a(t)<x<b(t)\}.
\]
Thus, the marginal game continues while the uncontrolled reserve remains
strictly between the two boundaries. Contact with $b$ corresponds to the
dividend stopping region, and contact with $a$ corresponds to the capital
injection stopping region.

\paragraph{Spatial smooth fit.}
Since the spatial dimension is one and $p>3$, parabolic Sobolev embedding
gives a locally continuous representative of $v_x$ on
$(0,T)\times(0,\infty)$, including across the two endogenous boundaries.
For fixed $t$, the function $v(t,\cdot)$ is constant on the interior of
each contact interval in \eqref{eq:contact-sets}. Continuity of $v_x$
therefore gives
\begin{equation}\label{eq:smooth-fit}
 v_x(t,b(t)-)=0,
 \qquad
 v_x(t,a(t)+)=0\quad\text{if }a(t)>0,
 \qquad 0<t<T.
\end{equation}
The second identity concerns an endogenous lower boundary inside the state
space. We make no smooth-fit assertion at the fixed boundary $x=0$.
Whenever an argument starts at $t=0$, we first start at a positive time and
then pass to the limit. Appendix~\ref{app:canonical-verification} proves continuity of $v_x$ at $t=0$ for
the canonical class.

\paragraph{The equation used in the verification argument.}
Proposition~\ref{prop:viscosity-characterization} establishes the obstacle equation \eqref{eq:double-obstacle} in the viscosity sense under Assumptions~\ref{assumption2_1}--\ref{assumption2_2}. Under \ref{ass:B2}, the equation also holds almost everywhere, in terms of weak derivatives, in the interior of each region. In the continuation region, the equation reduces to $\mathcal Lv=0$. In the two contact regions, $v=m$ and $v=f$, respectively. Hence
\[
 \mathcal Lv(t,x)=
 \begin{cases}
  m'(t),&0<x<a(t),\\[2pt]
  0,&a(t)<x<b(t),\\[2pt]
  f'(t),&x>b(t),
 \end{cases}
 \qquad\text{for almost every }(t,x) \in\mathcal Q^\circ.
\]
The two free-boundary curves
\[
\{(t,a(t)):0\le t<T\} \qquad \text{and}\qquad \{(t,b(t)):0\le t<T\}
\]
have two-dimensional Lebesgue measure zero, so the values of $\mathcal Lv$ on these curves do not affect this identity. This is the form of $\mathcal Lv$ substituted into the It\^o--Krylov formula in
Proposition~\ref{prop:v-representation} and later integrated in the spatial variable in Section \ref{sec:value-characterization}. The boundary
and terminal values remain those in \eqref{eq: game-bound}. Since $(T,0)$ is excluded from $\mathcal Q$, no compatibility between the prescribed solvency-boundary data and terminal data is required at that point.

\paragraph{The saddle point in boundary form.}
Proposition~\ref{prop:game-existence} gives the saddle point as the first times at which the game
value meets either obstacle. Using \eqref{eq:contact-sets}, these times can
now be written directly in terms of the two boundaries:
\begin{align}
 \tau^*(t,x)
 &=\inf\{s\in[0,\rho):A_s(x)\geq b(t+s)\}\wedge\rho,
 \label{eq:upper-hit}\\
 \delta^*(t,x)
 &=\inf\{s\in[0,\rho):A_s(x)\leq a(t+s)\}\wedge\rho.
 \label{eq:lower-hit}
\end{align}
Thus $\tau^*$ is the first entry of the uncontrolled reserve into the
dividend region, whereas $\delta^*$ is its first entry into the capital
injection region. In particular,
\[
 \mathbb E[\Gamma_{t,x}(\tau,\delta^*)]
 \leq \mathbb E[\Gamma_{t,x}(\tau^*,\delta^*)]
 =v(t,x)
 \leq \mathbb E[\Gamma_{t,x}(\tau^*,\delta)],
 \qquad \tau,\delta\in\mathcal T_\rho.
\]
These are stopping rules for the marginal game, whose state process is
$A(x)$. In Section~\ref{sec:value-characterization}, the same curves $a$ and $b$ become the lower and upper
reflecting boundaries for the controlled reserve process. This is the link
between the Dynkin game and the double-reflected control policy.

\section{From the Dynkin Game to the Optimal Dividend and Financing Policy}
\label{sec:value-characterization}

The game identifies the marginal value of reserves and the levels at which
the firm should intervene.  To recover the original dividend and capital injection control value itself, we first derive an
occupation-time representation of the game value, then construct the
reflected policy and verify its optimality.

\subsection{Representation and boundary equations}
\label{subsec:v-representation}

\begin{proposition}\label{prop:v-representation}
Under Assumptions~\ref{assumption2_1} and \ref{assumption3_1}, let
$t<T$, $x>0$, $H=T-t$, and $\rho=H\wedge S(x)$.  Then
\begin{align}
v(t,x)=\E\Bigg[&m(t+\rho)\ind_{\{S(x)<H\}}
 +g_x(T,A_H(x)^+)\ind_{\{H\le S(x)\}}\label{eq:v-representation}\\
&-\int_0^\rho f'(t+s)
       \ind_{\{A_s(x)\ge b(t+s)\}}\dd s\nonumber\\
&-\int_0^\rho m'(t+s)
       \ind_{\{0<A_s(x)\le a(t+s)\}}\dd s\Bigg].\nonumber
\end{align}
\end{proposition}

\begin{proof}
First, fix $t\in(0,T)$, and for $0<\varepsilon<H$, set $\rho_\varepsilon:=\rho\wedge(H-\varepsilon).$ For $n\ge1$, introduce the spatial localization time
\[
\zeta_n:=\rho_\varepsilon\wedge\inf\left\{s\ge0:A_s(x)\notin(\frac{1}{n},n)\right\}.
\]
Then, on $[0,\zeta_n]$, the process
$(t+s,A_s(x))$ remains in a compact subset of
$(0,T)\times(0,\infty)$.  Since
\[
v\in W^{1,2}_{p,\mathrm{loc}}
((0,T)\times(0,\infty)),
\qquad p>3,
\]
Assumption~\ref{assumption3_1}\ref{ass:B2} allows us to apply the generalized It\^o--Krylov formula (see, e.g., \citet{krylov2008controlled}) to $v(t+s,A_s(x))$.  Recalling that $\dd A_s(x)=\mu\dd s+\sigma\dd W_s,$
we obtain
\begin{align}\label{eq:ito-krylov-local}
v(t+\zeta_n,A_{\zeta_n}(x))-v(t,x)
={}&
\int_0^{\zeta_n}
\mathcal L v(t+s,A_s(x))\dd s+
\sigma\int_0^{\zeta_n}
v_x(t+s,A_s(x))\dd W_s,
\end{align}
where $\mathcal L=\partial_t+\mu\partial_x+(\sigma^2/2)\partial_{xx}$. By parabolic Sobolev embedding, $v_x$ has a continuous representative
and is bounded on the compact set before $\zeta_n$.  Hence, the second integral on the right-hand side of \eqref{eq:ito-krylov-local} is square integrable and has zero expectation.  Taking expectations on both sides of \eqref{eq:ito-krylov-local} yields
\begin{equation}\label{eq:v-ito-krylov-local}
\E\!\left[v(t+\zeta_n,A_{\zeta_n}(x))\right]-v(t,x)=\E\!\left[\int_0^{\zeta_n}\mathcal L v(t+s,A_s(x))\dd s
\right].
\end{equation}
Since $A_{\cdot}(x)$ has continuous paths, $\zeta_n\uparrow\rho_\varepsilon$ almost surely as $n\to\infty$. The double obstacle equation and the stopping region characterization imply that $v=m$ on the capital-injection region and $v=f$ on the dividend region, whereas $f<v<m$ in the continuation region.
Then, we have for almost every $(s,y)$,
\begin{equation}\label{eq:Lv-regions}
\mathcal Lv(t+s,y)
=
\begin{cases}
m'(t+s),
   & 0<y<a(t+s),\\[1mm]
0,
   & a(t+s)<y<b(t+s),\\[1mm]
f'(t+s),
   & y>b(t+s).
\end{cases}
\end{equation}

Local Sobolev regularity justifies the stopped
It\^o--Krylov formula, whose drift involves the combined
quantity $\mathcal Lv$. The preceding identity in \eqref{eq:Lv-regions}
yields
$|\mathcal Lv|\le K_L :=\max\{\|m'\|_\infty,\|f'\|_\infty\}$ almost everywhere.
Since the unstopped diffusion $A_s(x)$ has a Gaussian
density for every $s>0$, Tonelli's theorem implies that
this bound holds along $(t+s,A_s(x))$ for
$\mathrm ds\otimes\mathbb P$-almost every $(s,\omega)$.
Therefore,
\[
 \int_0^{\zeta_n}
 |(\mathcal Lv)(t+s,A_s(x))|\,\mathrm ds
 \le K_L\zeta_n\le K_LT
 \quad\text{almost surely}.
\]

Letting $n\to\infty$ in \eqref{eq:v-ito-krylov-local} and applying the dominated convergence theorem gives
\begin{equation}\label{eq:v-ito-krylov}
\E\!\left[
v(t+\rho_\varepsilon,A_{\rho_\varepsilon}(x))
\right]
-v(t,x)
=
\E\!\left[
\int_0^{\rho_\varepsilon}
\mathcal Lv(t+s,A_s(x))\dd s
\right].
\end{equation}
If $t=0$, apply the same argument starting at elapsed time $h>0$, with all processes stopped at $\rho_\varepsilon$, and let $h\downarrow0$; continuity and boundedness give the required identity.

For each fixed $s>0$, the random variable $A_s(x)$ has a density; hence, $\mathbb{P}(A_s(x)=a(t+s))=\mathbb{P} (A_s(x)=b(t+s))=0.$
Fubini's theorem then shows that, almost surely, $(t+s,A_s(x))$ lies on either deterministic boundary curve only for a set of times $s$ of Lebesgue measure zero.  Thus, the values assigned to the indicators at equality do not affect the time integrals in \eqref{eq:v-representation}.
Substituting \eqref{eq:Lv-regions} into
\eqref{eq:v-ito-krylov} and rearranging yields
\begin{align}
v(t,x)=\E\Bigg[v(t+\rho_\varepsilon,A_{\rho_\varepsilon}(x))& -\int_0^{\rho_\varepsilon}
f'(t+s)
\ind_{\{A_s(x)\ge b(t+s)\}}\dd s
\label{eq:v-representation-epsilon}\\
&-
\int_0^{\rho_\varepsilon}
m'(t+s)
\ind_{\{0<A_s(x)\le a(t+s)\}}\dd s
\Bigg].
\nonumber
\end{align}
It remains to let $\varepsilon\downarrow0$.  On the event
$\{S(x)<H\}$, one has $\rho_\varepsilon=S(x)$ for all sufficiently small
$\varepsilon$, and hence $v(t+\rho_\varepsilon,A_{\rho_\varepsilon}(x))=v(t+S(x),0)=m(t+S(x))$.
On $\{H<S(x)\}$, $A_{H-\varepsilon}(x)\longrightarrow A_H(x)>0$. Since $v\in C(\mathcal Q)$ with $\mathcal Q=([0,T]\times[0,\infty))\setminus\{(T,0)\}$, and $v(T,y)=g_x(T,y)$ for $y>0$, it follows that
\[
v(T-\varepsilon,A_{H-\varepsilon}(x))
\longrightarrow
g_x(T,A_H(x)),\qquad \text{as }\quad  \varepsilon\downarrow0.
\]
For $x>0$, the event $\{S(x)=H\}$ has probability zero, so the possible
incompatibility at the point $(T,0)$ does not affect the limiting expectation in \eqref{eq:v-representation}.  This is also consistent with the convention in \eqref{eq:Z} that gives priority to maturity.
Finally, $v$, $f'$ and $m'$ are bounded on the finite horizon.  Applying the dominated convergence theorem to \eqref{eq:v-representation-epsilon} as $\varepsilon\downarrow 0$ yields \eqref{eq:v-representation}.
\end{proof}

Under Assumption~\ref{assumption3_1}, the contact sets are closed with $v(t,b(t))=f(t)$, and $v(t,a(t))=m(t)$ whenever $a(t)>0$. Substituting $x=b(t)$ and $x=a(t)$, respectively, into \eqref{eq:v-representation} gives:
\begin{align}
f(t)
=
\E\Bigg[&
m\big(t+(H\wedge S(b(t)))\big)
\ind_{\{S(b(t))<H\}}
+
g_x\big(T,A_H(b(t))^+\big)
\ind_{\{H\le S(b(t))\}}
\label{eq:b-integral}\\
&\quad
-\int_0^{H\wedge S(b(t))}
f'(t+s)
\ind_{\{A_s(b(t))\ge b(t+s)\}}\dd s
\nonumber\\
&\quad
-\int_0^{H\wedge S(b(t))}
m'(t+s)
\ind_{\{0<A_s(b(t))\le a(t+s)\}}\dd s
\Bigg].
\nonumber
\end{align}
Whenever $a(t)>0$, the lower boundary satisfies
\begin{align}
m(t)
=
\E\Bigg[&
m\big(t+(H\wedge S(a(t)))\big)
\ind_{\{S(a(t))<H\}}
+
g_x\big(T,A_H(a(t))^+\big)
\ind_{\{H\le S(a(t))\}}
\label{eq:a-integral}\\
&\quad
-\int_0^{H\wedge S(a(t))}
f'(t+s)
\ind_{\{A_s(a(t))\ge b(t+s)\}}\dd s
\nonumber\\
&\quad
-\int_0^{H\wedge S(a(t))}
m'(t+s)
\ind_{\{0<A_s(a(t))\le a(t+s)\}}\dd s
\Bigg].
\nonumber
\end{align}

Equations~\eqref{eq:b-integral}--\eqref{eq:a-integral} form a coupled
system because the occupation terms in each equation depend on both moving
boundaries.  These equations give necessary conditions for the regular boundaries $a$ and $b$; uniqueness of their solutions requires further analysis. They also provide a useful probabilistic description of the two moving boundaries. When $a(t)=0$, Eq. \eqref{eq:a-integral} reduces to the exogenous solvency boundary condition and does not supply a separate integral equation.

\subsection{The reflected policy and the value function}

Extend $a,b$ continuously to $T$ by $a(T)=a_T$, $b(T)=b_T$.
Assumption~\ref{assumption3_1} gives
\[
\inf_{0\le s\le T}(b(s)-a(s))>0.
\]
For initial state $(t,x)$, set $H=T-t$.  The time-dependent Skorokhod
problem defines a controlled reserve process and two nondecreasing regulators by
\begin{equation}\label{eq:skorokhod}
X_s=x+\mu s+\sigma W_s-D_s^b+I_s^a
       \in[a(t+s),b(t+s)],\qquad 0\le s\le H,
\end{equation}
with initial intervention
\begin{equation}\label{eq:projection}
\Delta I_0^a=(a(t)-x)^+,\qquad
\Delta D_0^b=(x-b(t))^+,
\end{equation}
and complementarity conditions
\begin{equation}\label{eq:complementarity}
\int_{(0,H]}(X_s-a(t+s))\dd I_s^a=0,\qquad
\int_{(0,H]}(b(t+s)-X_s)\dd D_s^b=0.
\end{equation}
Existence, pathwise uniqueness, and finite variation follow from the
time-dependent Skorokhod map; see, e.g.,
\citet[Theorem~2.6 and Corollary~2.4]{burdzy2009skorokhod}.
The map is nonanticipating, so the controls are adapted.  Since the
driving process and both boundaries are continuous, the regulators are
continuous after the possible initial intervention, including at maturity. Finally, the integrability required for admissibility (see Definition \ref{def:admissible}) is established in the following proposition.

\begin{proposition}\label{prop:sk-admissible}
Under Assumption~\ref{assumption3_1}, the reflected control policy $(D^b,I^a)\in\A(t,x)$.
\end{proposition}
\begin{proof}
Only integrability remains to be proved; the other admissibility conditions follow from the construction above. In particular, the complementarity conditions in \eqref{eq:complementarity} imply that, after a possible initial intervention, $D^b$ can increase only when $X_s=b(t+s)$, whereas $I^a$ can increase only when $X_s=a(t+s)$.  Since $a(t+s)<b(t+s)$, the two regulators cannot increase simultaneously.  The initial intervention in \eqref{eq:projection} is admissible because at most one of the two initial jumps is positive, and the post-intervention state belongs to $[a(t),b(t)]\subset[0,\infty)$.

The uniform continuity of $a$ and $b$, together with their positive separation, gives a finite partition $0=s_0<\cdots<s_J=H$ such that, on each $[s_j,s_{j+1}]$, the ranges of
$a(t+\cdot)$ and $b(t+\cdot)$ are disjoint compact sets.  Choose a
$C^2$ function $h_j$ with bounded derivatives such that $h_j'=1$ on the
first range and $h_j'=-1$ on the second.  It\^o's formula and
\eqref{eq:complementarity} give
\begin{align*}
I^a_{s_{j+1}}-I^a_{s_j}+D^b_{s_{j+1}}-D^b_{s_j}
={}&h_j(X_{s_{j+1}})-h_j(X_{s_j})\\
&-\int_{s_j}^{s_{j+1}}
 \left(\mu h_j'(X_s)+\frac{\sigma^2}{2}h_j''(X_s)\right)\dd s
 -\sigma\int_{s_j}^{s_{j+1}}h_j'(X_s)\dd W_s.
\end{align*}
The right-hand side is integrable, since $X$ remains in a deterministic bounded
interval and the stochastic integrand is bounded.  Summing and adding
the finite initial intervention proves $\E[D_H^b+I_H^a]<\infty$.
\end{proof}

Define the payoff of this policy by
\begin{equation}\label{def:Mtx}
M(t,x):=J(t,x;D^b,I^a).
\end{equation}
We now state the main theorem.
\begin{theorem}\label{thm:theorem-main}
Under Assumptions~\ref{assumption2_1}, \ref{assumption2_2}, and
\ref{assumption3_1}, the reflected policy $(D^b,I^a)$ is optimal and
$V(t,x)=M(t,x)$ for every $t<T$ and $x\ge0$.  Moreover,
\[
V_x(t,x)=v(t,x)\quad(x>0),\qquad V_x(t,0+)=m(t),
\]
and
\begin{align}
V(t,x)&=V(t,b(t))-\int_x^{b(t)}v(t,y)\dd y,\label{V-v-b}\\
V(t,x)&=V(t,a(t))+\int_{a(t)}^xv(t,y)\dd y.\label{V-v-a}
\end{align}
The additive term $V(t,b(t))$ is determined explicitly by
\begin{equation}\label{eq:upper-boundary-value-main}
V(t,b(t))=g(T,b_T)+\mu\int_t^T f(s)\dd s
  +f(t)b(t)-f(T)b_T+\int_t^T f'(s)b(s)\dd s,
\end{equation}
and then
\begin{equation}\label{eq:lower-boundary-value-main}
V(t,a(t))=V(t,b(t))-\int_{a(t)}^{b(t)}v(t,y)\dd y.
\end{equation}
At maturity, $V(T,x)=g(T,x)$.
\end{theorem}

The remainder of this section proves Theorem~\ref{thm:theorem-main}.  By Proposition~\ref{prop:sk-admissible},  the reflected strategy $(D^b,I^a)$ is admissible, so its payoff satisfies $M\le V$.
 To prove the reverse inequality $M\ge V$, we construct an auxiliary function $N$ by the spatial integral of the game value $v$. We first show that $N$ dominates the payoff of every admissible dividend and capital injection control policy, so $N\ge V$. We then prove that the reflected policy attains $N$, giving $N=M$ and hence $M=V=N$. The argument builds on the integration approach of \citet{karoui2006integration} and \citet{ferrari2019optimal}, with the marginal value now represented by a Dynkin game and the lower reflection boundary determined endogenously.

\subsection{Verification}
\label{subsec:verification}
Set
\begin{equation}\label{eq:B}
B(t):=g(T,b_T)+\mu\int_t^T f(s)\dd s
 +f(t)b(t)-f(T)b_T+\int_t^T f'(s)b(s)\dd s,
\end{equation}
and define the auxiliary function
\begin{equation}\label{eq:N}
N(t,x):=B(t)-\int_x^{b(t)}v(t,y)\dd y,
\qquad t<T,\quad x\ge0.
\end{equation}
In particular, $N(t,b(t))=B(t)$ and
\[
N(t,x)=B(t)+f(t)(x-b(t)),\qquad x\ge b(t).
\]
We shall prove $M=N=V$.  We first establish weak derivative identities for $N$, then prove its classical $C^{1,2}$ regularity in the interior. We apply It\^o's formula to the spatially shifted process and pass to the limit to include the solvency boundary and the initial time.

\begin{lemma}\label{lem:N-basic}
The function $N$ is continuous on $[0,T)\times[0,\infty)$, satisfies a uniform
linear growth bound, and belongs locally to $W^{1,2}_p$ on
$(0,T)\times(0,\infty)$.  Its spatial derivative satisfies
\begin{equation}\label{eq:Nx-v}
N_x(t,x)=v(t,x),
\end{equation}
with the right derivative used at $x=0$, and hence
\begin{equation}\label{eq:N-gradient-bounds}
f(t)\le N_x(t,x)\le m(t).
\end{equation}
Furthermore,
\begin{equation}\label{eq:N-terminal-limit}
N(t,\cdot)\longrightarrow g(T,\cdot)
\quad\text{locally uniformly on }[0,\infty)\text{ as }t\uparrow T.
\end{equation}
\end{lemma}
\begin{proof}
The fundamental theorem of calculus gives $N_x=v$, including the right
derivative at zero, and the obstacle bounds give
\eqref{eq:N-gradient-bounds}.  Continuity follows from the continuity
of $B$, $b$, and $v$ before maturity.  Since $B,b,f,m$ are bounded on
the finite horizon, there is a constant $C>0$ such that
\[
|N(t,x)|\le C(1+x),\qquad t<T,\quad x\ge0.
\]

For the weak time derivative, we fix a compact rectangle contained in the interior of the domain and choose
$R>\sup_{0\le u\le T}b(u)$, also larger than the spatial coordinates in
that rectangle.  Since $v(t,y)=f(t)$ for $y\ge b(t)$, the definition of
$N$ can be rewritten as
\begin{equation}\label{eq:N-fixed-anchor}
N(t,x)=C_R(t)-\int_x^R v(t,y)\dd y,
\end{equation}
where
\begin{equation}\label{eq:CR}
C_R(t)=g(T,b_T)+\mu\int_t^T f(s)\dd s-f(T)b_T
          +\int_t^T f'(s)b(s)\dd s+f(t)R.
\end{equation}
Here, $C_R(\cdot)$ is continuously differentiable and
\[
C_R'(t)=-\mu f(t)+f'(t)(R-b(t)).
\]
Assumption~\ref{assumption3_1}\ref{ass:B2} and Fubini's theorem now give
the local Sobolev regularity of $N$, with weak derivatives
\[
N_t(t,x)=C_R'(t)-\int_x^R v_t(t,y)\dd y,
\qquad N_{xx}=v_x.
\]

For terminal convergence, first observe from \eqref{eq:B} that
$B(t)\to g(T,b_T)$ when $t \uparrow T$.  Given $K>0$, take $R$ also larger than $K$.
The continuity of $v$ at $(T,y)$ for $y>0$ and its uniform boundedness imply, by the dominated convergence theorem, that
\[
\int_0^R|v(t,y)-g_x(T,y)|\dd y\longrightarrow0.
\]
The possible discrepancy at the single point $(T,0)$ does not affect
this integral. Then, uniformly for $x\in[0,K]$, we have
\begin{align*}
\left|\int_x^{b(t)}v(t,y)\dd y
          -\int_x^{b_T}g_x(T,y)\dd y\right|
&\le\int_0^R|v(t,y)-g_x(T,y)|\dd y\\
&\quad+\|g_x(T,\cdot)\|_\infty|b(t)-b_T|.
\end{align*}
The right-hand side tends to zero.  Substituting into \eqref{eq:N}
and using
$g(T,b_T)-\int_x^{b_T}g_x(T,y)\dd y=g(T,x)$ proves
\eqref{eq:N-terminal-limit}, including at $x=0$.
\end{proof}

\begin{proposition}\label{prop:hjb-N}
The expression $\mathcal LN$, defined in terms of weak derivatives, has the continuous representative
\begin{equation}\label{eq:LN-representative}
F(t,x):=-m'(t)(a(t)-x)^++f'(t)(x-b(t))^+,
\qquad t<T,\quad x\ge0.
\end{equation}
The function $F$ is nonpositive and satisfies $|F(t,x)|\le C(1+x)$ for $t<T$ and $x\ge0$.  In particular, the
following identities hold almost everywhere in the indicated regions:
\begin{align}
\mathcal LN(t,x)&=0,&&a(t)<x<b(t),\label{eq:LNcont}\\
\mathcal LN(t,x)&=f'(t)(x-b(t)),&&x>b(t),\label{eq:LNupper}\\
\mathcal LN(t,x)&=m'(t)(x-a(t)),&&0<x<a(t).
\label{eq:LNlower}
\end{align}
Therefore, the variational inequality
\begin{equation}\label{eq:NVI}
\max\{\mathcal LN,\ f(t)-N_x,\ N_x-m(t)\}=0
\end{equation}
holds almost everywhere on $(0,T)\times(0,\infty)$.  With $F$ in place
of $\mathcal LN$, it holds at every $(t,x)\in[0,T)\times[0,\infty)$.
\end{proposition}
\begin{proof}
Use the identity in \eqref{eq:N-fixed-anchor} for $N(t,x)$.  Since $v(t,R)=f(t)$ and
$v_x(t,R)=0$, integration of
$\mathcal Lv=v_t+\mu v_x+(\sigma^2/2)v_{xx}$ gives
\begin{align}
\mathcal LN(t,x)
 &=C_R'(t)-\int_x^R v_t(t,y)\dd y
                  +\mu v(t,x)+\frac{\sigma^2}{2}v_x(t,x)\nonumber\\
 &=C_R'(t)+\mu f(t)-\int_x^R\mathcal Lv(t,y)\dd y\nonumber\\
 &=f'(t)(R-b(t))-\int_x^R\mathcal Lv(t,y)\dd y.
\label{eq:LN-fixed-integral}
\end{align}
 We evaluate its last integral separately in the three regions below.

If $a(t)<x<b(t)$, the interval $(x,R)$ consists of continuation states
followed by the dividend region.  By \eqref{eq:Lv-regions},
\[
\int_x^R\mathcal Lv(t,y)\dd y=f'(t)(R-b(t)).
\]
This cancels the first term in \eqref{eq:LN-fixed-integral} and proves
\eqref{eq:LNcont}.
If $x>b(t)$, the whole interval is in the dividend
region, so the same integral equals $f'(t)(R-x)$.  Subtraction gives
\eqref{eq:LNupper}.
Finally, if $0<x<a(t)$, the integral also includes
the injection region, and
\[
\int_x^R\mathcal Lv(t,y)\dd y
   =m'(t)(a(t)-x)+f'(t)(R-b(t)).
\]
This proves \eqref{eq:LNlower}.  Since the rectangle and $R$ were
arbitrary, the identities hold locally throughout the open domain.

The three expressions agree at the boundaries and combine to give
\eqref{eq:LN-representative}.  Continuity of $a,b,f',m'$ makes $F$
continuous, including at $x=0$ and $t=0$.  Its growth is at most
linear because the boundaries and the derivatives of $f$ and $m$ are bounded.
Proposition~\ref{prop:intervention-sign} gives $f'(t)\le0$ and, whenever
$a(t)>0$, $m'(t)\ge0$.  Hence $F\le0$.  In the continuation interval
$F=0$, while in each intervention region the corresponding gradient
constraint holds.  Value matching and continuity extend the identity to the free boundaries.
\end{proof}

\begin{proposition}
\label{cor:N-classical}
Under Assumptions~\ref{assumption2_1} and \ref{assumption3_1}, we have $N\in C^{1,2}((0,T)\times(0,\infty))$ and $\mathcal LN=F$ pointwise in that open domain.
\end{proposition}

\begin{proof}
By Assumption~\ref{assumption3_1}\ref{ass:B2} and the local
parabolic Sobolev embedding in one spatial dimension, $p>3$
implies that $v$ is classically differentiable in $x$ in the
open domain and that $v_x$ has a locally continuous representative.
Hence, by \eqref{eq:Nx-v}, $N_x=v$ and $N_{xx}=v_x$ pointwise on $(0,T)\times(0,\infty)$. In particular, both spatial
derivatives of $N$ are continuous there.

Proposition~\ref{prop:hjb-N} now gives, in the weak sense,
\[
N_t=G:=F-\mu v-\frac{\sigma^2}{2}v_x.
\]
The function $G$ is continuous in the open domain. It remains to
show that this weak time derivative is also the classical derivative.

Fix $0<t_0<t_1<T$ and  $0<x_0<x_1<\infty$, and set
\[
Q'=(t_0,t_1)\times(x_0,x_1).
\]
By Fubini's theorem and the one-dimensional Sobolev fundamental
theorem of calculus, for almost every $x\in(x_0,x_1)$, the map
$t\mapsto N(t,x)$ is absolutely continuous and
\[
N(s_2,x)-N(s_1,x)
   =\int_{s_1}^{s_2}G(r,x)\dd r,
\qquad t_0<s_1<s_2<t_1.
\]
The absolutely continuous representative agrees with
$N(\cdot,x)$, since $N$ is continuous. For fixed $s_1,s_2$,
both sides of the  identity are continuous in $x$.
Consequently, the identity extends from almost every
$x\in(x_0,x_1)$ to every $x\in(x_0,x_1)$.

It follows that $t\mapsto N(t,x)$ is continuously differentiable
for every $x\in(x_0,x_1)$, with $N_t(t,x)=G(t,x)$.
Since $G$ is jointly continuous and $Q'$ was arbitrary,
$N_t$ is continuous throughout the open domain. Therefore
$
N\in C^{1,2}((0,T)\times(0,\infty)),
$
and substitution of the classical derivatives into the definition
of $\mathcal L$ yields $\mathcal LN=F$ pointwise.
\end{proof}

\begin{lemma}\label{lem:ito-krylov-jumps}
For any $(D,I)\in\A(t,x)$ and $0\le s<H$, write $X=X^{D,I}$ and
$X_{0-}=x$.  Then
\begin{align}
N(t+s,X_s)-N(t,x)
={}&\int_0^s F(t+r,X_{r-})\dd r
  +\sigma\int_0^s N_x(t+r,X_{r-})\dd W_r\label{eq:ito-jump}\\
&-\int_{(0,s]}N_x(t+r,X_{r-})\dd D_r^c
 +\int_{(0,s]}N_x(t+r,X_{r-})\dd I_r^c\nonumber\\
&+\sum_{0\le r\le s}
       \bigl(N(t+r,X_r)-N(t+r,X_{r-})\bigr).\nonumber
\end{align}
Here $D^c,I^c$ denote the continuous parts, and $N_x(t,0)$ is the
right derivative $m(t)$.  The jump series is absolutely convergent.
\end{lemma}

\begin{proof}
By Proposition~\ref{cor:N-classical}, we have $N\in C^{1,2}((0,T)\times(0,\infty))$, and $\mathcal LN=F$
in the interior.  Moreover, Lemma~\ref{lem:N-basic} and
Proposition~\ref{prop:hjb-N} show that $N$, $N_x$, and $F$
are continuous up to $x=0$ for $t<T$, and that
\[
 |N_x(u,y)|\le K:=\max_{0\le q\le T}m(q),
 \qquad
 |N(u,y)|+|F(u,y)|\le C(1+y).
\]
Assume first that $s>0$. Fix $0<h<s$ and $\delta>0$, and define $X_r^\delta:=X_r+\delta$. Since $(D,I)$ is admissible, $X_r\ge0$ for $0\le r\le H$; hence $X_r^\delta\ge\delta$, and likewise $X_{r-}^\delta\ge\delta$. Moreover, because $H=T-t$, for $r\in[h,s]$,

$$ 0<t+h\le t+r\le t+s<T. $$
Thus, $(t+r,X_r^\delta)$ remains in $(0,T)\times(0,\infty)$, where $N$ is of class $C^{1,2}$. After spatial localization, the standard It\^o formula for c\`adl\`ag semimartingales therefore applies to $N(t+r,X_r^\delta)$ on $[h,s]$. In addition, since
\[
\dd X_r^\delta
 =\mu\,\dd r+\sigma\,\dd W_r-\dd D_r+\dd I_r
\]
and $\mathcal LN=F$, combining the linear jump terms with the
usual jump correction gives
\begin{align}
&N(t+s,X_s^\delta)-N(t+h,X_h^\delta)\nonumber\\
&\quad=\int_h^s F(t+r,X_{r-}^\delta)\dd r
 +\sigma\int_h^s N_x(t+r,X_{r-}^\delta)\dd W_r
 \nonumber\\
&\qquad-\int_{(h,s]}N_x(t+r,X_{r-}^\delta)\dd D_r^c
 +\int_{(h,s]}N_x(t+r,X_{r-}^\delta)\dd I_r^c
 \nonumber\\
&\qquad+\sum_{h<r\le s}
 \bigl(N(t+r,X_r^\delta)-N(t+r,X_{r-}^\delta)\bigr).
\label{eq:ito-shifted}
\end{align}

Let $\delta\downarrow0$. Pathwise boundedness of $X$, together
with the continuity of $N$, $N_x$, and $F$, gives convergence
of the endpoint, drift, and continuous finite-variation terms.
For the stochastic integrals, boundedness of $N_x$ and It\^o's
isometry give convergence in $L^2$. For every jump,
\[
\left|N(t+r,X_r^\delta)-N(t+r,X_{r-}^\delta)\right|
 \le K|\Delta X_r|
 \le K(\Delta D_r+\Delta I_r).
\]
Since
\[
\sum_{0\le r\le s}(\Delta D_r+\Delta I_r)
 \le D_s+I_s<\infty
 \qquad\text{a.s.},
\]
the dominated convergence theorem also applies to the jump series. Thus
\eqref{eq:ito-shifted} holds with $X$ in place of $X^\delta$.

It remains to let $h\downarrow0$. Right continuity gives
\[
N(t+h,X_h)\longrightarrow N(t,X_0).
\]
With the atom at zero assigned to the jump part,
$D_0^c=I_0^c=0$, and the continuous-control terms over
$(0,h]$ vanish. Moreover,
\[
\int_0^h |F(t+r,X_{r-})|\dd r\longrightarrow0
\]
pathwise, while
\[
\E\left|
 \sigma\int_0^h N_x(t+r,X_{r-})\dd W_r
 \right|^2
 \le \sigma^2K^2h\longrightarrow0.
\]
Finally,
\[
\sum_{0<r\le h}
 \left|N(t+r,X_r)-N(t+r,X_{r-})\right|
 \le K\bigl[(D_h-D_0)+(I_h-I_0)\bigr]
 \longrightarrow0.
\]
We therefore obtain the asserted identity with initial value
$N(t,X_0)$ and jumps in $(0,s]$. Adding
\[
N(t,X_0)-N(t,x)
   =N(t,X_0)-N(t,X_{0-})
\]
inserts the possible initial intervention and gives the sum over
$0\le r\le s$.

The same estimate also proves absolute convergence, i.e.,
\[
\sum_{0\le r\le s}
 \left|N(t+r,X_r)-N(t+r,X_{r-})\right|
 \le K(D_s+I_s)<\infty.
\]

For the expectation arguments used below, admissibility gives
\begin{equation}\label{eq:state-integrable-bound}
\sup_{0\le r\le H}X_r
 \le x+|\mu|H+\sigma\sup_{0\le r\le H}|W_r|+I_H
 =:\Xi,
\qquad \E[\Xi]<\infty.
\end{equation}
Together with the preceding growth bounds, this makes the endpoint
and drift terms integrable.
\end{proof}

\begin{lemma}\label{lem:L1terminal}
Fix $(t,x)\in[0,T)\times[0,\infty)$ and put $H=T-t$. For every
$(D,I)\in\A(t,x)$,
\begin{equation}\label{eq:L1terminal}
N(T-\varepsilon,X_{H-\varepsilon}^{D,I})
 \longrightarrow g(T,X_{H-}^{D,I})
 \quad\text{in }L^1
 \qquad\text{as }\varepsilon\downarrow0,
\end{equation}
where $0<\varepsilon<H$.
\end{lemma}

\begin{proof}
Write $X=X^{D,I}$. Since $D$ and $I$ have finite left limits at
$H$ and $W$ is continuous, $X_{H-\varepsilon}\to X_{H-}$~ a.s. Moreover, the state constraint and \eqref{eq:state-integrable-bound}
give $ 0\le X_{H-\varepsilon},X_{H-}\le\Xi$~ a.s. For every sample path for which $\Xi<\infty$,
\begin{align*}
&|N(T-\varepsilon,X_{H-\varepsilon})
        -g(T,X_{H-})|\\
&\quad\le
 \sup_{0\le y\le\Xi}
 |N(T-\varepsilon,y)-g(T,y)|
 +\|g_x(T,\cdot)\|_\infty
   |X_{H-\varepsilon}-X_{H-}|.
\end{align*}
The first term tends to zero by the locally uniform terminal
convergence in Lemma~\ref{lem:N-basic}, and the second tends to zero
by the almost-sure convergence of the state. Hence the asserted
convergence holds almost surely.

Finally, the uniform linear-growth bound for $N$ and the boundedness
of $g_x(T,\cdot)$ imply
\[
|N(T-\varepsilon,X_{H-\varepsilon})|
 +|g(T,X_{H-})|
 \le C(1+\Xi).
\]
Since $\E[\Xi]<\infty$, the dominated convergence theorem yields the $L^1$ convergence in \eqref{eq:L1terminal}.
\end{proof}

We can now prove Theorem~\ref{thm:theorem-main}.

\paragraph{Proof of Theorem~\ref{thm:theorem-main}.}
Fix $(t,x)\in[0,T)\times[0,\infty)$ and put $H=T-t$.
Admissibility of the reflected policy gives
\begin{equation}\label{eq:MleV}
M(t,x)\le V(t,x).
\end{equation}

To prove the reverse inequality, let $(D,I)\in\A(t,x)$ be arbitrary.
Write $X=X^{D,I}$ and, for $0<\varepsilon<H$, set
$H_\varepsilon=H-\varepsilon$. Applying \eqref{eq:ito-jump} at
$H_\varepsilon$ and adding the accumulated control payoff gives
\begin{align*}
&\int_{[0,H_\varepsilon]}f(t+r)\dd D_r
 -\int_{[0,H_\varepsilon]}m(t+r)\dd I_r
 +N(T-\varepsilon,X_{H_\varepsilon})-N(t,x)\\
&=\int_0^{H_\varepsilon}F(t+r,X_{r-})\dd r
 +\sigma\int_0^{H_\varepsilon}
       N_x(t+r,X_{r-})\dd W_r\\
&\quad+\int_{(0,H_\varepsilon]}
       (f-N_x)(t+r,X_{r-})\dd D_r^c\\
&\quad+\int_{(0,H_\varepsilon]}
       (N_x-m)(t+r,X_{r-})\dd I_r^c\\
&\quad+\sum_{0\le r\le H_\varepsilon}
 \Bigl[
 N(t+r,X_r)-N(t+r,X_{r-})
 +f(t+r)\Delta D_r-m(t+r)\Delta I_r
 \Bigr].
\end{align*}

The drift term is nonpositive because $F\le0$, and the two
continuous-control terms are nonpositive by
\eqref{eq:N-gradient-bounds}. For a jump at time $r$, put $z=X_{r-}$, $d=\Delta D_r$, and  $i=\Delta I_r$. Feasibility gives $z-d+i\ge0$. If $i\ge d$, the bound $N_x\le m$
gives
\[
N(t+r,z-d+i)-N(t+r,z)+f(t+r)d-m(t+r)i
 \le \bigl(f(t+r)-m(t+r)\bigr)d.
\]
If $d>i$, the bound $N_x\ge f$ gives the corresponding estimate with
$i$ on the right. Thus, in both cases,
\begin{equation*}
N(t+r,z-d+i)-N(t+r,z)+f(t+r)d-m(t+r)i
 \le\bigl(f(t+r)-m(t+r)\bigr)\min\{d,i\}\le0.
\end{equation*}
This calculation also applies to the initial jump, for which
$z=X_{0-}=x$.

All terms above are integrable by
\eqref{eq:state-integrable-bound}, the linear growth of $N$ and $F$,
the boundedness of $N_x,f,m$, and admissibility. In particular, the
stochastic integral is a true martingale. Taking expectations gives
\[
\E\left[
 \int_{[0,H_\varepsilon]}f(t+r)\dd D_r
 -\int_{[0,H_\varepsilon]}m(t+r)\dd I_r
 +N(T-\varepsilon,X_{H_\varepsilon})
 \right]\le N(t,x).
\]

The convergence of the control payoffs can be seen directly from
\begin{align*}
\E\left|
 \int_{[0,H)}f(t+r)\dd D_r
 -\int_{[0,H_\varepsilon]}f(t+r)\dd D_r
\right|
&\le\|f\|_\infty
   \E[D_{H-}-D_{H_\varepsilon}]
 \longrightarrow0,
\end{align*}
and analogously
\[
\E\left|
 \int_{[0,H)}m(t+r)\dd I_r
 -\int_{[0,H_\varepsilon]}m(t+r)\dd I_r
\right|
\le\|m\|_\infty
   \E[I_{H-}-I_{H_\varepsilon}]
\longrightarrow0.
\]
Indeed, the differences converge to zero almost surely and are
dominated by the integrable variables $D_H$ and $I_H$.
Lemma~\ref{lem:L1terminal} therefore yields
\[
\E\left[
 \int_{[0,H)}f(t+r)\dd D_r
 -\int_{[0,H)}m(t+r)\dd I_r
 +g(T,X_{H-})
 \right]\le N(t,x).
\]

It remains to include a possible intervention at maturity. Since $X_H=X_{H-}-\Delta D_H+\Delta I_H$, the terminal gradient bounds
$f(T)\le g_x(T,\cdot)\le m(T)$ give, by the same jump calculation,
\[
g(T,X_H)-g(T,X_{H-})
 +f(T)\Delta D_H-m(T)\Delta I_H\le0.
\]
Adding this inequality and using
\[
\int_{[0,H]}f(t+r)\dd D_r
 =\int_{[0,H)}f(t+r)\dd D_r+f(T)\Delta D_H,
\]
together with the analogous identity for $I$, proves $J(t,x;D,I)\le N(t,x)$. Since $(D,I)$ was arbitrary, we have
\begin{equation}\label{eq:VleN}
V(t,x)\le N(t,x).
\end{equation}

Next, we consider the reflected policy $(D^b,I^a)$. Its state remains in $[a(t+r),b(t+r)]$, where $F=0$. The regulators are continuous on
$(0,H]$, and the complementarity condition \eqref{eq:complementarity} gives
\[
N_x(t+r,X_{r-})=f(t+r)
   \quad\dd D_r^b\text{-a.e.},\qquad N_x(t+r,X_{r-})=m(t+r)
   \quad\dd I_r^a\text{-a.e.}
\]
on $(0,H)$. Hence, both continuous-control contributions vanish.

The possible initial intervention also has zero net contribution.
Indeed, if $x<a(t)$, then
\[
N(t,a(t))-N(t,x)-m(t)(a(t)-x)=0,
\]
while, if $x>b(t)$,
\[
N(t,b(t))-N(t,x)+f(t)(x-b(t))=0.
\]
There are no subsequent jumps. Consequently, the preceding
preterminal inequality is an equality for the reflected policy.
Letting $\varepsilon\downarrow0$ gives
\[
\E\left[
 \int_{[0,H)}f(t+r)\dd D_r^b
 -\int_{[0,H)}m(t+r)\dd I_r^a
 +g(T,X_{H-})
 \right]=N(t,x).
\]
Since the regulators are continuous at $H$, they have no atom at
maturity and $X_H=X_{H-}$. Thus
\[
M(t,x)=N(t,x).
\]

Combining this equality with \eqref{eq:MleV} and
\eqref{eq:VleN} gives $M(t,x)=V(t,x)=N(t,x)$. The derivative and boundary identities follow from \eqref{eq:N}.
At $x=0$, continuity of $v(t,\cdot)$ and $v(t,0)=m(t)$ give the
right derivative. Finally, Proposition~\ref{prop:terminal-envelope}
gives $V(T,x)=g(T,x)$. This completes the proof of Theorem~\ref{thm:theorem-main}. \hfill $\square$

\section{A Canonical Class with Proactive Capital Injection}
\label{sec:canonical}
The verification results in Section~\ref{sec:value-characterization} rely on Assumption~\ref{assumption3_1}.
We now verify this assumption for a specific class of financing costs and terminal payoffs. This yields a concrete class for which the required boundary regularity follows from explicit assumptions
on the model primitives.

Let $r>0$ and $\eta>0$, and set
\begin{equation}\label{eq:canonical-f}
 f(t)=\eta e^{-rt},\qquad m(t)=e^{-rt}\kappa(t),
 \qquad 0\le t\le T,
\end{equation}
where $\kappa\in C^2([0,T])$ satisfies
\begin{equation}\label{eq:canonical-kappa}
 \kappa(t)>\eta,\qquad \kappa'(t)-r\kappa(t)>0,
 \qquad 0\le t\le T.
\end{equation}
Thus, $f'(t)<0$ and $m'(t)>0$: the marginal dividend payment decreases in valuation units, whereas the marginal financing cost increases.
The second condition in \eqref{eq:canonical-kappa} requires growth in the financing costs to
dominate discounting.

Fix $0\le A_\star<B_\star<\infty$ and prescribe the terminal
marginal payoff by
\begin{equation}\label{eq:canonical-gx}
 g_x(T,x)=
 \begin{cases}
 m(T),&0\le x\le A_\star,\\[1mm]
 m(T)-\dfrac{m(T)-f(T)}{B_\star-A_\star}(x-A_\star),
       &A_\star<x<B_\star,\\[2mm]
 f(T),&x\ge B_\star.
 \end{cases}
\end{equation}
For a finite constant $g(T,0)\ge 0$, define
\begin{equation}\label{eq:canonical-g}
 g(T,x)=g(T,0)+\int_0^x g_x(T,y)\dd y.
\end{equation}
This specification gives a concave terminal reserve value
with a nonincreasing marginal value of liquidity.
The marginal value equals the terminal financing cost on
$[0,A_\star]$, decreases linearly on $(A_\star,B_\star)$,
and equals the marginal dividend reward for reserves at
or above $B_\star$.
The terminal payoff may be interpreted as a continuation
value assigned to liquidity at the end of the planning horizon.

The terminal marginal payoff in \eqref{eq:canonical-gx} is bounded, nonnegative, nonincreasing, and Lipschitz continuous. Thus, $g(T,\cdot)$ is
continuously differentiable, increasing, and concave, with
$f(T)\le g_x(T,x)\le m(T)$. Its limiting marginal value as $x\to \infty$ is $f(T)$,
and
\[
 \max\left\{\sup_{t\le s\le T}f(s),\,
                  \lim_{x\to\infty}g_x(T,x)\right\}
 =f(t)<m(t).
\]
Therefore, Assumptions~\ref{assumption2_1}
and~\ref{assumption2_2} hold.

\begin{theorem}
\label{thm:canonical-regularity}
Under \eqref{eq:canonical-f}--\eqref{eq:canonical-g}, the game value
$v$ is continuous on $[0,T]\times[0,\infty)$ and, for every
$1<p<\infty$,
\begin{equation}\label{eq:canonical-v-regularity}
 v\in W^{1,2}_{p,\mathrm{loc}}
           \big((0,T)\times(0,\infty)\big).
\end{equation}
The associated boundaries are finite and continuous on $[0,T)$,
with
\begin{equation}\label{eq:canonical-ab}
 0\le a(t)<b(t)<\infty,\qquad 0\le t<T,
\end{equation}
and terminal limits
\begin{equation}\label{eq:terminal-ab}
 \lim_{t\uparrow T}a(t)=A_\star,
 \qquad \lim_{t\uparrow T}b(t)=B_\star.
\end{equation}
Hence, Assumption~\ref{assumption3_1} holds. The control value and the optimal reflected policy are given by Theorem~\ref{thm:theorem-main}.
\end{theorem}
\begin{proof}
Lemma~\ref{lem:app-regularity} verifies
Assumption~\ref{assumption3_1}\ref{ass:B2}; in fact, it proves the
required Sobolev regularity for every $1<p<\infty$, as well as
continuity on the closed state--time domain.
Lemma~\ref{lem:app-displacement} proves that both contact levels are
finite and provides the one-sided comparison estimates needed for
their temporal regularity. Lemma~\ref{lem:app-continuity} then proves
that $a$ and $b$ are continuous on $[0,T)$, strictly separated, and
satisfy
\[
a(T-)=A_\star,\qquad b(T-)=B_\star.
\]
Thus Assumption~\ref{assumption3_1}\ref{ass:B1} holds with
$a_T=A_\star$ and $b_T=B_\star$. Since $A_\star<B_\star$,
Assumption~\ref{assumption3_1} is fully verified, and the remaining conclusions follow from Theorem~\ref{thm:theorem-main}.

\end{proof}

Extend the boundaries continuously by $a(T)=A_\star$ and
$b(T)=B_\star$. Compactness gives
\begin{equation}\label{eq:canonical-uniform-gap}
 \min_{0\le t\le T}\bigl(b(t)-a(t)\bigr)>0.
\end{equation}
Thus, the reflected strategy is admissible by
Proposition~\ref{prop:sk-admissible}, and the integral representation
and coupled boundary equations in
Proposition~\ref{prop:v-representation} apply.

\section{Numerical Analysis}\label{sec:numerical}

In this section, we solve the marginal Dynkin game numerically on a truncated spatial domain. For the canonical examples, we also recover the value of the dividend and capital-injection problem using Theorem~\ref{thm:theorem-main}.
Throughout the numerical examples, we normalize the terminal payoff by
setting $g(T,0)=0$.
The parameter values are chosen to illustrate the model\textquotesingle s mechanisms and are not calibrated to a particular firm, industry, or data set.

Let $R=X_{\max}>B_\star$ and introduce the uniform grid
\[
    x_i=i\Delta x,\qquad i=0,\ldots,N_x,
    \qquad
    t_n=n\Delta t,\qquad n=0,\ldots,N_t,
\]
where
\[
    \Delta x=\frac{R}{N_x},
    \qquad
    \Delta t=\frac{T}{N_t}.
\]
For brevity, write $f_n=f(t_n)$ and $m_n=m(t_n)$. The boundary and
terminal conditions are
\[
    v_0^n=m_n,\qquad
    v_{N_x}^n=f_n,
    \qquad n=0,\ldots,N_t-1,
\]
and
\[
    v_i^{N_t}=g_x(T,x_i),
    \qquad i=0,\ldots,N_x,
\]
where the terminal data satisfy
\[
    f(T)\leq g_x(T,x_i)\leq m(T).
\]
The left boundary represents absorption at the solvency level, while
the artificial right boundary is chosen sufficiently far inside the
dividend region.

Starting from the terminal data at $t=T$, we compute $v^n$ backward in time.
At each interior node, the fully implicit double-obstacle problem is
\begin{equation*}
    \max\left\{
        v_i^n-m_n,\,
        \min\left[
            v_i^n-f_n,\,
            \frac{v_i^n-v_i^{n+1}}{\Delta t}
            -(\mathcal L_hv^n)_i
        \right]
    \right\}=0,
    \qquad i=1,\ldots,N_x-1,
\end{equation*}
where $\mathcal L_h$ is the centered finite-difference approximation of the spatial operator $\mu\partial_x + (\sigma^2/2)\partial_{xx}$. Hence, the continuation equation holds whenever
$f_n<v_i^n<m_n$, while the lower and upper obstacles identify the
dividend and financing regions, respectively. The resulting discrete
complementarity problem is solved by projected successive
over-relaxation; implementation details are provided in
Appendix~\ref{app:numerical-details}.

At each preterminal time point, the numerical financing and dividend
boundaries are identified by
\begin{align*}
    a_h(t_n)
    &=
    \max\left\{
        x_i:
        v_j^n\geq m_n-\varepsilon_c
        \text{ for every }j\leq i
    \right\},
    \\
    b_h(t_n)
    &=
    \min\left\{
        x_i:
        v_j^n\leq f_n+\varepsilon_c
        \text{ for every }j\geq i
    \right\}.
\end{align*}

For the canonical class, the value of the original control problem is
reconstructed from the numerical marginal value. Its value at the
dividend boundary is
\begin{align}
    B_h(t_n)
    ={}&
    g(T,B_\star)
    +\mu\int_{t_n}^{T}f(s)\,\dd s
    +f(t_n)b_h(t_n)-f(T)B_\star
    \notag\\
    &\quad
    +\int_{t_n}^{T}f'(s)b_h(s)\,\dd s,
    \label{eq:numerical-value-anchor}
\end{align}
and the value at each spatial grid point is recovered through
\begin{equation}
\label{eq:numerical-value-reconstruction}
    V_h(t_n,x_i)
    =
    B_h(t_n)
    -
    \int_{x_i}^{b_h(t_n)}v_h(t_n,y)\,\dd y.
\end{equation}
Thus, $B_h(t_n)$ supplies the additive normalization that cannot be
determined from the marginal game value alone. Numerical integration,
boundary classification, truncation control, and consistency checks are
described in Appendix~\ref{app:numerical-details}.

\subsection{Baseline results}
\label{subsec:numerical-illustrations}

The baseline parameters are
\[
T=5,\qquad \mu=0.05,\qquad \sigma=0.15,\qquad
r=0.05,\qquad \eta=1,\qquad \kappa(t)=1.5+t,
\]
with $A_\star=0.5$, $B_\star=2$, and $g(T,0)=0$.
These parameters satisfy the canonical assumptions of
Section~\ref{sec:canonical}. In particular,
\[
m'(t)=e^{-0.05t}(0.925-0.05t)>0,
\qquad t\in[0,5].
\]
For $R=5$, $N_x=800$, and $N_t=1600$, the computed
boundaries satisfy
\[
a_h(0)=1.53125,\qquad b_h(0)=2.04375,
\]
and
\[
\min_n\bigl(b_h(t_n)-a_h(t_n)\bigr)=0.51250,
\qquad
\max_n b_h(t_n)=2.17500.
\]

Figure~\ref{fig:canonical-baseline} displays the computed boundaries
in panel~(a) and the reconstructed value function for the dividend
and capital-injection problem in panel~(b). When the reserve lies
below $a(t)$, the firm injects capital to restore it to $a(t)$.
When the reserve exceeds $b(t)$, the excess is distributed as
dividends. No intervention occurs while the reserve remains between
the two boundaries. The resulting policy reflects both the high
marginal terminal value of reserves at low reserve levels and the
increasing path of financing costs.

\begin{figure}[t]
\centering
\includegraphics[width=\textwidth]
{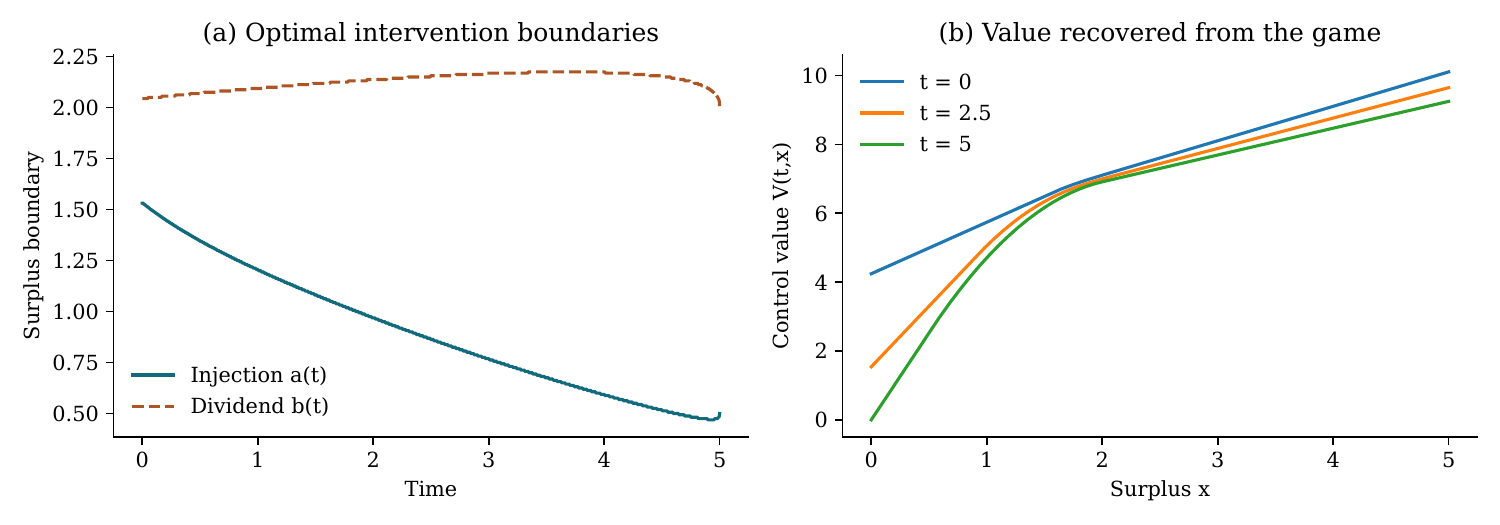}
\caption{Baseline results: (a) optimal intervention boundaries over time; (b) the reconstructed shareholder value at $t=0$, $2.5$, and $5$.}
\label{fig:canonical-baseline}
\end{figure}

\begin{table}[t]
\centering
\caption{Baseline grid refinement on the spatial domain $[0,5]$, with $T=5$. The evaluation point
$x=1.8$ lies in the initial continuation region.}
\label{tab:canonical-refinement}
\begin{tabular}{rrrrrr}
\toprule
$N_x$ & $N_t$ & $a_h(0)$ & $b_h(0)$
& $v_h(0,1.8)$ & $V_h(0,1.8)$\\
\midrule
200 & 400  & 1.52500 & 2.05000 & 1.1306167 & 6.8935967\\
400 & 800  & 1.53750 & 2.05000 & 1.1306285 & 6.8938616\\
800 & 1600 & 1.53125 & 2.04375 & 1.1306051 & 6.8939790\\
\bottomrule
\end{tabular}
\end{table}

Table~\ref{tab:canonical-refinement} shows that the value estimates are
stable under grid refinement. The extracted boundaries exhibit
variation at approximately the scale of the spatial mesh, as expected
for grid-based boundary identification. Every grid specification
satisfies the obstacle bounds and the required spatial monotonicity.

We also vary the boundary-detection tolerance between
$\varepsilon_c=10^{-8}$ and $10^{-6}$. On the finest grid (i.e., $N_x=800, N_t=1600$), each
reported boundary changes by at most one spatial step, namely
$0.00625$. To separate truncation effects from mesh effects, we keep $\Delta x=0.00625$ and $N_t=1600$ fixed and use $(R,N_x)=(4,640),(5,800),(6,960)$. All three grids give $a_h(0)=1.53125$ and $b_h(0)=2.04375$.

As an independent numerical check, we simulate a discrete-projection
approximation of the reflected strategy, starting from the initial state $(t,x)= (0,1.8)$,
using $50{,}000$ sample paths. At each simulation date, the reserve
is first updated by the Brownian surplus increment and then projected
onto the interpolated continuation band. The associated dividend and
capital-injection cash flows are recorded directly, after which the
terminal payoff is added.
Using $500$, $1000$, and $2000$ time steps gives estimated payoffs
of $6.89356$, $6.89441$, and $6.89377$, respectively, with Monte
Carlo standard errors of $0.00176$, $0.00177$, and $0.00176$.
All three estimates are statistically consistent with the reconstructed
finite-difference value $V_h(0,1.8)=6.893979$.

\Needspace{6\baselineskip}
\subsection{Comparisons with a fixed terminal payoff}\label{subsec:compareterminal}

We next vary these illustrative model primitives on a grid with $N_x=600$ and
$N_t=1200$, while holding the baseline terminal payoff fixed. The
drift and volatility parameters take values
\[
\mu\in\{0.02,0.05,0.08\},
\qquad
\sigma\in\{0.15,0.20,0.30\}.
\]
For the financing-cost experiments, define $m_0(t)=e^{-0.05t}(1.5+t)$ and consider
\[
m_\alpha(t)
=
m_0(t)+\alpha(1-t/T),
\qquad
\alpha\in\{-0.3,0,0.3\},
\]
and
\[
m_\beta(t)
=
m_0(t)+\beta\frac{t(T-t)}{T},
\qquad
\beta\in\{-0.4,0,0.4\}.
\]
Both perturbation families preserve $m(T)$, and therefore leave the
terminal payoff and its compatibility conditions unchanged. The
$\alpha$-family changes the current financing cost, with the
perturbation tapering linearly to zero at maturity. The
$\beta$-family preserves both the initial and terminal financing
costs while modifying financing conditions at intermediate dates.

All the perturbed cost functions satisfy $m>f$ and $m'>0$. Indeed,
for the baseline path,
\[
\min_{t\in[0,5]}m_0'(t)
=
0.675e^{-0.25}
>
0.525.
\]
The derivative of the $\beta$-perturbation has absolute value at most
$0.4$, while that of the $\alpha$-perturbation has absolute value
at most $0.06$. Hence, all the financing-cost paths remain strictly
increasing. Moreover, each can be written as
$m(t)=e^{-rt}\kappa(t)$, where $\kappa\in C^2$ satisfies
\eqref{eq:canonical-kappa}. Theorem~\ref{thm:canonical-regularity}
therefore applies to every experiment in this subsection. The results
are reported in Figure~\ref{fig:canonical-comparative-statics} and
Table~\ref{tab:canonical-comparative-statics}.

\begin{figure}[htt]
\centering
\includegraphics[width=\textwidth]
{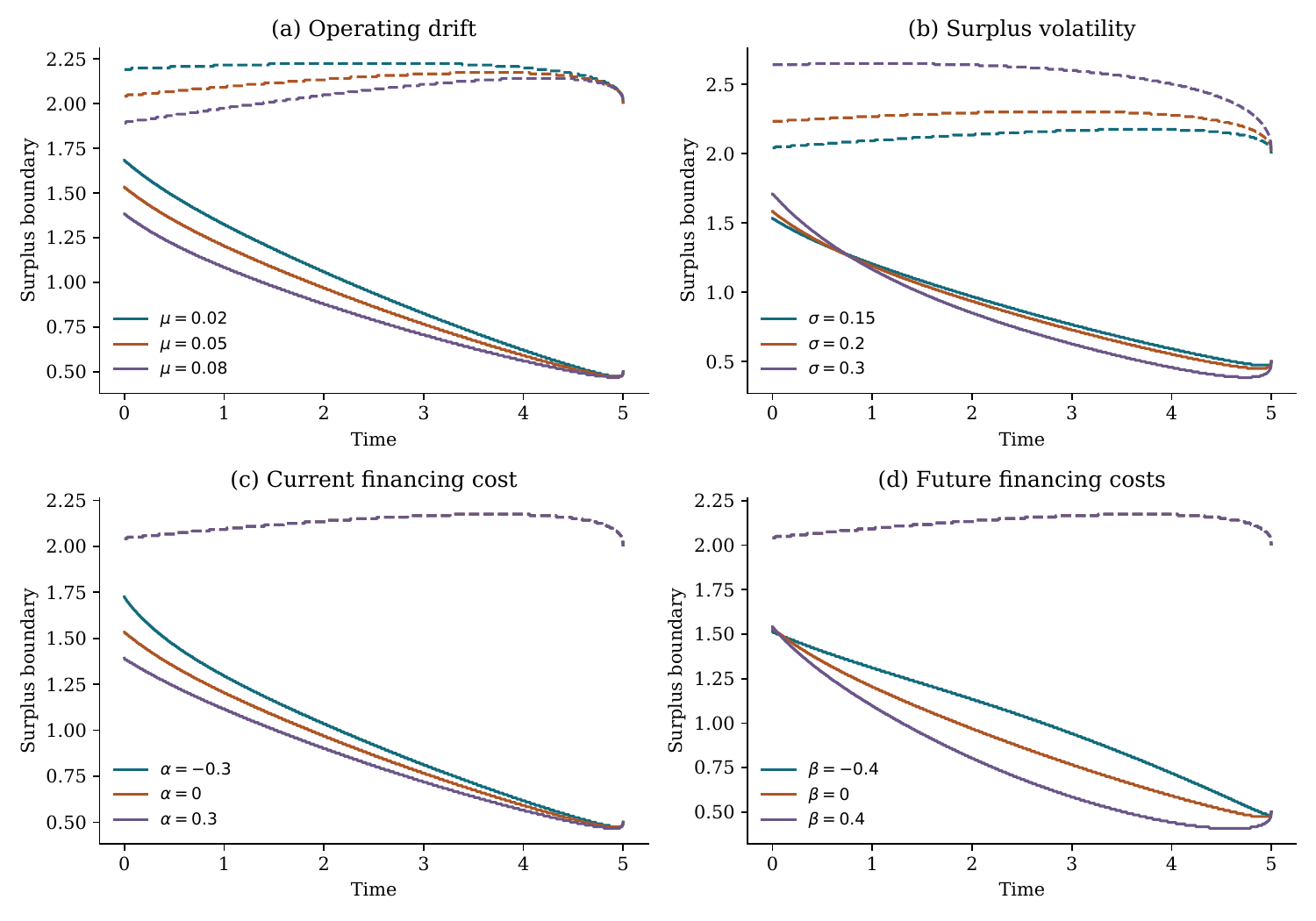}
\caption{Comparative statics with a common terminal payoff. Solid
curves represent the lower boundary $a(t)$, and dashed curves
represent the upper boundary $b(t)$. The financing-cost
perturbations are defined in Section \ref{subsec:compareterminal}.}
\label{fig:canonical-comparative-statics}
\end{figure}

\begin{table}[ht]
\centering
\caption{Initial-boundary comparisons with a fixed terminal payoff.
Each row reports the results at the lowest and highest parameter values, with boundaries extracted on the spatial grid.}
\label{tab:canonical-comparative-statics}
\begin{tabular}{lrrr}
\toprule
Parameter change
& $a_h(0)$ & $b_h(0)$ & $b_h(0)-a_h(0)$\\
\midrule
$\mu:0.02\to0.08$
& $1.6833\to1.3833$
& $2.1917\to1.8917$
& $0.5083\to0.5083$\\
$\sigma:0.15\to0.30$
& $1.5333\to1.7083$
& $2.0417\to2.6417$
& $0.5083\to0.9333$\\
$\alpha:-0.3\to0.3$
& $1.7250\to1.3917$
& $2.0417\to2.0417$
& $0.3167\to0.6500$\\
$\beta:-0.4\to0.4$
& $1.5167\to1.5417$
& $2.0417\to2.0417$
& $0.5250\to0.5000$\\
\bottomrule
\end{tabular}
\end{table}

Economically, stronger operating inflows reduce the need to hold reserves; hence, a higher drift shifts both initial intervention boundaries downward. This is consistent with a spatial translation of the marginal problem. In the canonical setting, fix the other primitives, choose a reference drift $\mu_0$, put $\delta=\mu-\mu_0$, and extend $v^{\mu_0}(t,y)$ by $m(t)$ for $y<0$. If $a^{\mu_0}(s)>0$ and $a^{\mu_0}(s)-\delta(T-s)>0$ for all $0\le s<T$, then the change of variable $y=x+\delta(T-t)$ preserves the double-obstacle equation, the terminal data and the boundary value at zero. Uniqueness therefore gives $a^\mu(t)=a^{\mu_0}(t)-\delta(T-t)$, and $ b^\mu(t)=b^{\mu_0}(t)-\delta(T-t)$. Both boundaries then move by the same amount, as illustrated in Table~\ref{tab:canonical-comparative-statics}.
 In addition, higher volatility raises both the initial capital injection and
dividend boundaries and widens the continuation band, with a
substantially larger response at the dividend boundary.

The financing-cost experiments admit a more specific interpretation. In these experiments, increasing $\alpha$ raises the current cost of capital injection and substantially lowers the computed initial injection boundary, consistent with a weaker incentive to issue equity immediately. The corresponding change in the initial dividend boundary is smaller than the spatial grid resolution.

In the $\beta$-experiment, the initial and terminal financing costs
are unchanged, whereas a larger $\beta$ makes financing more expensive
at all intermediate dates. Proposition~\ref{prop:cost-comparison}
therefore gives the relevant ordering of the intervention boundaries.
The computed increase in $a_h(0)$ is consistent with an incentive to
prefund more expensive future capital needs. The numerical effect is
considerably smaller than in the $\alpha$-experiment at the initial time. The two experiments separate
the effect of expensive current issuance from the incentive to prefund
more expensive future issuance, while keeping the incentives generated by the terminal payoff fixed.

\subsection{Passive and nonmonotone financing costs}

We finally examine financing-cost paths outside the canonical class. These examples describe the stopping regions of the marginal game. Their interpretation as reflected control policies requires Assumption~\ref{assumption3_1}. Keeping $T,\mu,\sigma,r,\eta$ at their baseline values, we consider
\[
m_\gamma(t)=e^{-rt}(1.5+\gamma t),
\qquad
\gamma\in\{0,0.09\}.
\]
Both cases use the same terminal marginal payoff,
\[
G(x)
=
f(T)+\bigl(1.5e^{-rT}-f(T)\bigr)(1-x/2)^+,
\]
together with the normalization $g(T,0)=0$.

The case $\gamma=0$ corresponds to a constant nominal financing cost, with $m_0'(t)=-1.5r e^{-rt}<0$. For $\gamma=0.09$, $m_\gamma'(t)
=e^{-rt}\bigl(\gamma-r(1.5+\gamma t)\bigr)$, so the discounted financing cost first increases and then decreases, although the nominal cost $1.5+0.09t$ increases throughout the horizon. Its turning
date is
\[
t^*
=
\frac{1}{r}-\frac{1.5}{\gamma}
=
3.3333.
\]
Moreover, $G(0)=1.5e^{-rT}$ equals $m_0(T)$ when $\gamma=0$, but is strictly below
$m_{0.09}(T)$ when $\gamma=0.09$. Hence, the common terminal payoff
is compatible with both financing-cost specifications.

Figure~\ref{fig:beyond-canonical} reports the resulting optimal stopping
boundaries. In the decreasing-cost case, the positive-reserve capital-injection region is empty, with $a_h(t)=0$ for all $t$. The nonmonotone example further shows that
a locally increasing financing cost, $m_\gamma'(t)>0$ for
$t<t^*$, does not by itself imply the existence of a positive
capital-injection boundary over that interval.  Hence, $m'(t)\ge 0$ is a necessary, but not sufficient, condition for proactive capital injection.

\begin{figure}[t]
\centering
\includegraphics[width=\textwidth]
{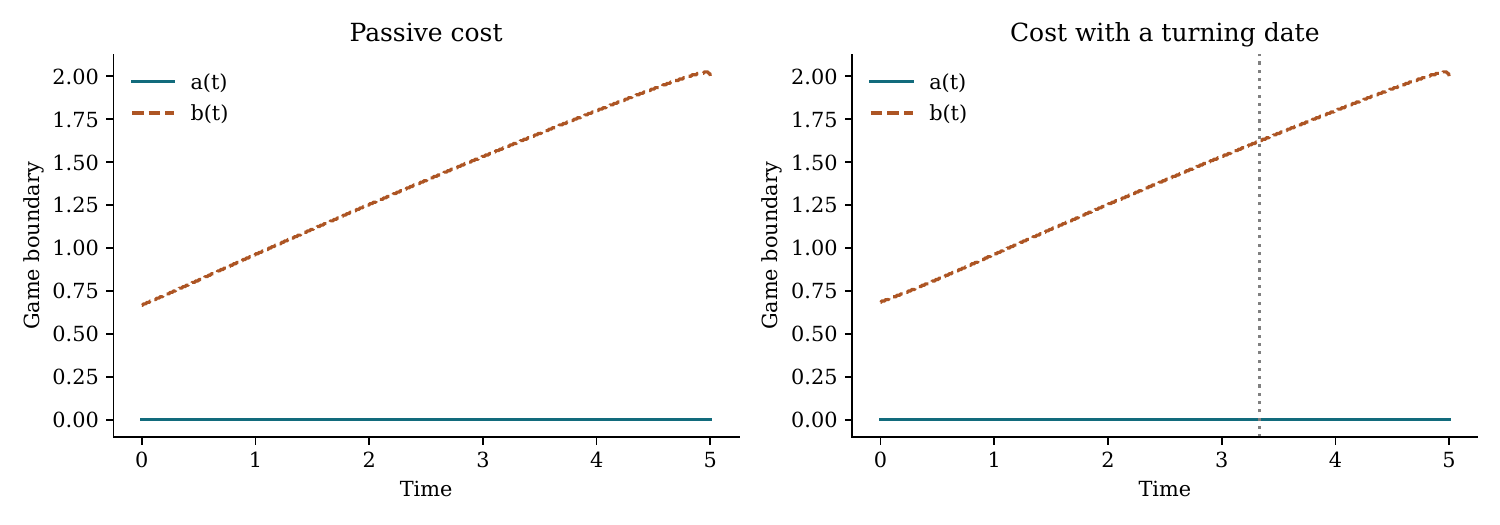}
\caption{Optimal stopping boundaries for a decreasing financing cost
and for a financing cost that first increases and then decreases. Both
panels use the same terminal payoff. The vertical dotted line marks
the turning date $t^*=3.3333$.}
\label{fig:beyond-canonical}
\end{figure}

\section{Conclusion}\label{sec:conclusion}

This paper studies how time-varying financing conditions alter the joint
management of dividends and capital injections over a finite horizon. The
main departure from the classical formulation is that capital injection is
an independently timed financing decision rather than an intervention
triggered mechanically by insolvency. When future external financing is
expected to become more expensive, shareholders may optimally inject
capital while the firm still has positive reserves. Under the verification conditions, the optimal policy has two endogenous components: a lower financing boundary and an upper
dividend boundary. Reserves are increased below the former, reduced above
the latter, and left uncontrolled between them.

The marginal value of liquidity provides the link between this economic
mechanism and the mathematical structure of the problem. We characterize
that marginal value through a finite-horizon Dynkin game whose two contact
regions determine the financing and dividend boundaries. This
representation also yields a necessary sign condition for proactive financing. If $a(t)$ denotes the lower financing boundary, then
\[
    a(t)>0 \quad \Longrightarrow \quad m'(t)\geq 0.
\]
Thus, at any date $t$ with $m'(t)< 0$, proactive injection is ruled out at that date. If this strict inequality holds throughout the horizon, the financing boundary remains at zero. An increasing
financing cost, however, is not sufficient by itself: the financing
decision also depends on operating risk, the remaining horizon, the
dividend reward, and the terminal value assigned to liquidity. The
transition from a one-player stopping problem under passive injection to a
Dynkin game under proactive injection therefore reflects a substantive
change in the underlying financing decision.

Under a regular intervention-band condition, we recover the shareholder
value by integrating the game value with respect to reserves. We then verify optimality over general admissible
c\`adl\`ag singular controls and show that the optimal policy is the
two-sided Skorokhod reflection between the financing and dividend
boundaries. In addition, we establish the regular intervention-band condition for a canonical class of primitives with decreasing marginal dividend rewards, increasing marginal financing
costs, and a compatible piecewise affine terminal marginal value.

The comparative statics clarify the distinction between total shareholder
value and the marginal value of liquidity. Higher anticipated financing
costs weakly reduce shareholder value because future recapitalization becomes
more expensive, but they weakly increase the game value, which equals the marginal value of reserves under the verification conditions. Consequently, a higher financing-cost path weakly raises the dividend boundary, encouraging the firm to retain more liquidity. When the
current financing cost is held fixed, higher future costs also weakly raise the current financing boundary, strengthening the incentive to recapitalize
before financing conditions deteriorate. The numerical experiments
illustrate these effects while holding the terminal payoff fixed, and
display both proactive- and passive-injection regimes as well as
nonmonotone financing-cost environments.

Several extensions would broaden the scope of the analysis. One direction
is to develop primitive boundary results for standard linear liquidation
payoffs and for intervention bands that collapse at maturity. These cases
fall outside the regular terminal geometry considered here and may require
different local estimates near the horizon. Another direction is to allow
financing conditions to be stochastic, in which case the financing
environment would become an additional state variable, and the intervention
regions would be multidimensional. Extensions to more general surplus dynamics, fixed financing costs, or endogenous investment opportunities
would further connect the model to corporate liquidity management. The
present analysis provides a benchmark for these developments by showing
how anticipated financing conditions can transform capital injection from
a passive solvency mechanism into a proactive component of optimal
liquidity policy.

%\ifdefined\anonymize
%
%\else
%\section*{Declaration of interests}
%The authors declare that they have no competing interests.
%\section*{Acknowledgments}
%Ran Xu acknowledges support from the National Natural Science Foundation of China under Grant Nos. 12201506 and 12371468.
%\fi

\appendix

\section{Regularity for the canonical specification}
\label{app:canonical-verification}
This appendix verifies Assumption \ref{assumption3_1} for the canonical specification in
Section \ref{sec:canonical}. We retain the notation of Section \ref{sec:game}, in particular
\[
 A_s(x)=x+\mu s+\sigma W_s,\qquad
 S(x)=\inf\{s\geq0:A_s(x)\leq0\},\qquad
 \mathcal L=\partial_t+\mu\partial_x+\frac{\sigma^2}{2}\partial_{xx}.
\]
Write $G(x)=g_x(T,x)$ and set
\[
 c=\min_{0\leq t\leq T}\min\{-f'(t),m'(t)\}>0,
 \qquad
 \Delta_0=\max_{0\leq t\leq T}\bigl(m(t)-f(t)\bigr).
\]
Recall that $G$ is bounded, Lipschitz continuous, and nonincreasing. It equals
$m(T)$ on $[0,A_\star]$ and $f(T)$ on $[B_\star,\infty)$, and is strictly
decreasing between these two levels. Throughout this appendix, $C$ denotes a positive constant that depends only on the model data
and may change from line to line.

\begin{lemma}\label{lem:app-regularity}
The canonical game value $v$ is continuous on $[0,T]\times[0,\infty)$ and
belongs to $W^{1,2}_{p,\mathrm{loc}}((0,T)\times(0,\infty))$ for every
$1<p<\infty$. Its spatial derivative is locally continuous on
$[0,T)\times(0,\infty)$. Moreover, for some $C<\infty$,
\[
 \sup_{x\geq0}|v(T-h,x)-G(x)|\leq C(h+\sqrt h),
 \qquad 0<h\leq T.
\]
\end{lemma}

\begin{proof}
We first estimate how close $v(t,x)$ is to $G(x)$ when $t$ is close to $T$.
Let $C_g$ denote the Lipschitz constant of $G$. Fix
$t=T-h$ and $x>0$, and put $\rho=h\wedge S(x)$. If the minimizing player
chooses $\delta=\rho$, any earlier dividend payment is at most
$G(x)+\|f'\|_\infty h$. The maturity payoff differs from $G(x)$ by at most
$C_g|A_h(x)-x|$. On absorption, the payment is
$m(t+S(x))\leq m(T)=G(0)$, and
$x\leq\sup_{r\leq h}|A_r(x)-x|$. Thus, taking the supremum over the other
player's stopping times gives
\[
 v(t,x)-G(x)
 \leq \|f'\|_\infty h
 +C_g\,
   \mathbb E\!\left[\sup_{r\leq h}|\mu r+\sigma W_r|\right].
\]
For the reverse estimate, let the maximizing player choose $\tau=\rho$.
Every capital injection or absorption payment is at least
$m(T)-\|m'\|_\infty h\geq G(x)-\|m'\|_\infty h$, while the maturity
payoff has the same Lipschitz bound. Consequently,
\[
 v(t,x)-G(x)
 \geq-\|m'\|_\infty h
 -C_g\,
   \mathbb E\!\left[\sup_{r\leq h}|\mu r+\sigma W_r|\right].
\]
The Brownian maximal estimate bounds the expectation by $C(h+\sqrt h)$,
where $C$ is independent of $x$ and $h$. At $x=0$,
\[
 |v(T-h,0)-G(0)|=|m(T-h)-m(T)|\leq\|m'\|_\infty h.
\]
Combining these inequalities gives
\[
 \sup_{x\geq0}|v(T-h,x)-G(x)|\leq C(h+\sqrt h),
 \qquad 0<h\leq T.
\]
Thus $v(t,x)$ converges to $G(x)$ as $t\uparrow T$, with the same error
bound for every $x\geq0$. The bound also holds for $h=0$ because
$v(T,x)=G(x)$.

The same waiting strategy for the maximizing player also gives, for $t<T$,
\[
 m(t)-\Delta_0\,\mathbb P\bigl(S(x)\geq T-t\bigr)
 \leq v(t,x)\leq m(t).
\]
Indeed, $m$ is increasing, so only survival to maturity can produce a payment
below $m(t)$. For every fixed $\delta\in(0,T)$, it follows that
\[
 \sup_{0\leq t\leq T-\delta}|v(t,x)-m(t)|
 \leq\Delta_0\,\mathbb P(S(x)\geq\delta)
 \longrightarrow0\qquad\text{as }x\downarrow0.
\]
Since $m$ is continuous, this proves continuity of $v$ at each point
$(t,0)$ with $t<T$.

We also need continuity at $(T,0)$, where time tends to $T$ and the reserve
tends to zero simultaneously. The preceding estimate near $T$ and the
Lipschitz continuity of $G$ give
\[
 \begin{split}
 |v(t,x)-G(0)|
 &\leq |v(t,x)-G(x)|+|G(x)-G(0)|\\
 &\leq C\bigl((T-t)+\sqrt{T-t}\bigr)+C_gx
 \longrightarrow0\qquad\text{as }(t,x)\to(T,0).
 \end{split}
\]
Because $v(T,0)=G(0)=m(T)$, this proves continuity of $v$ at $(T,0)$.

We next prove the asserted Sobolev regularity on
$(0,T)\times(0,\infty)$ by penalization. For
$0<\varepsilon\leq1$, let $v^\varepsilon$ be the bounded continuous solution
on the half-line of
\[
 -\mathcal Lv^\varepsilon
 +\frac{(v^\varepsilon-m)^+}{\varepsilon}
 -\frac{(f-v^\varepsilon)^+}{\varepsilon}=0,
 \qquad
 v^\varepsilon(t,0)=m(t),\quad v^\varepsilon(T,x)=G(x).
\]
For fixed $\varepsilon>0$, the diffusion coefficient is strictly positive
and the penalty terms are Lipschitz continuous in $v^\varepsilon$.
Standard existence and regularity results for semilinear parabolic
boundary-value problems therefore give a bounded solution in
\[
 C\bigl([0,T]\times[0,\infty)\bigr)
 \cap C^{1,2}\bigl((0,T)\times(0,\infty)\bigr).
\]
The solution on $[0,\infty)$ can be obtained by first
solving on bounded intervals, imposing the value $f(t)$ at a right endpoint
beyond $B_\star$, and then letting that endpoint tend to infinity; the
comparison bounds below are independent of the endpoint. Uniqueness
follows from comparison. The conditions at $x=0$ and $t=T$ prescribe the
same value at $(T,0)$, since $G(0)=m(T)$.

Choose $C>\max\{\|f'\|_\infty,\|m'\|_\infty\}$. Substitution and comparison
with $f-C\varepsilon$ and $m+C\varepsilon$ yield
\[
 f-C\varepsilon\leq v^\varepsilon\leq m+C\varepsilon,
 \qquad |\mathcal Lv^\varepsilon|\leq C.
\]
For the second inequality, each penalty is bounded by $C$, and the two
penalties cannot be positive at the same point.

There is also a direct comparison with the game value. The function
$v^\varepsilon-C\varepsilon$ is a subsolution of the double-obstacle
equation \eqref{eq:double-obstacle}: it is at most $m$, and wherever it exceeds $f$ the lower
penalty vanishes, giving $-\mathcal Lv^\varepsilon\leq0$. Similarly,
$v^\varepsilon+C\varepsilon$ is a supersolution: it is at least $f$, and
wherever it lies below $m$ the upper penalty vanishes, giving
$-\mathcal Lv^\varepsilon\geq0$. The first function is below the
prescribed values $m(t)$ at $x=0$ and $G(x)$ at $t=T$, while the second is above them.
The corner ordering follows from $G(0)=m(T)$ and the terminal estimate established above.
Proposition \ref{prop:viscosity-characterization} and the comparison principle in Lemma \ref{lem:halfline-comparison} therefore give
\[
 v^\varepsilon-C\varepsilon\leq v\leq
 v^\varepsilon+C\varepsilon.
\]
In particular,
\[
 \sup_{\substack{0<t<T\\x>0}}|v^\varepsilon(t,x)-v(t,x)|
 \leq C\varepsilon.
\]
As a uniform limit of continuous functions, $v$ is therefore continuous
on $(0,T)\times(0,\infty)$.

For any cylinder $K$ with compact closure in $(0,T)\times(0,\infty)$,
the interior $L^p$ estimate for the linear operator $\mathcal L$, applied
on a slightly larger cylinder, now yields
\[
 \|v^\varepsilon\|_{W^{1,2}_p(K)}\leq C_{p,K},
 \qquad 1<p<\infty,
\]
with $C_{p,K}$ independent of $\varepsilon$. Passing to the limit proves
$v\in W^{1,2}_{p,\mathrm{loc}}((0,T)\times(0,\infty))$. For $p>3$,
parabolic Sobolev embedding in one spatial dimension gives a continuous
representative of $v_x$ on $(0,T)\times(0,\infty)$; see, for example,
\cite{krylov2008lectures} for the linear estimates and embedding results.

Finally, extend $f$ and $m$ smoothly to the left of zero, preserving
$f<m$ and $f'<0<m'$. The preceding argument applies to the game with these
extended time data. Its value agrees with the original value at every
starting time $t\geq0$, because only the data on $[t,T]$ enter the game.
The regularity argument now applies in a neighborhood of each $(0,x)$
with $x>0$, proving continuity of $v$ and $v_x$ at these points as well.
Together with the continuity already proved at $x=0$ and at $t=T$, this
establishes $v\in C([0,T]\times[0,\infty))$ and continuity of $v_x$ on
$[0,T)\times(0,\infty)$.
\end{proof}

\begin{lemma}\label{lem:app-displacement}
Set $a(T)=A_\star$ and $b(T)=B_\star$. Both boundaries are finite for $t<T$
and, for some constant $C>0$, satisfy
\begin{equation}\label{eq:app-displacement}
 \begin{split}
 a(t)&\geq a(s)-C(s-t)^{1/4},\\
 b(t)&\leq b(s)+C(s-t)^{1/4},
 \end{split}
 \qquad 0\leq t<s\leq T.
\end{equation}
In particular,
\begin{equation}\label{eq:app-terminal-bound}
 a(t)\geq A_\star-C(T-t)^{1/4},
 \qquad b(t)\leq B_\star+C(T-t)^{1/4}.
\end{equation}
\end{lemma}

\begin{proof}
For any stopping time $\zeta\leq h\leq T$, the elementary inequality $|u+v|^4\le 8(|u|^4+|v|^4)$
gives
\[
 \sup_{0\le r\le\zeta}|\mu r+\sigma W_r|^4
 \le 8|\mu|^4\zeta^4
     +8\sigma^4\sup_{0\le r\le\zeta}|W_r|^4.
\]
Applying the Burkholder--Davis--Gundy (BDG) inequality at the bounded stopping time $\zeta$ yields
\[
 \mathbb E\!\left[\sup_{0\le r\le\zeta}|W_r|^4\right]
 \le C_{\rm BDG}\,\mathbb E\langle W\rangle_\zeta^2
 =C_{\rm BDG}\,\mathbb E[\zeta^2].
\]
Combining these estimates and absorbing the fixed coefficients into a constant $K$, we obtain
\[
 \mathbb E\!\left[\sup_{0\le r\le\zeta}
                    |\mu r+\sigma W_r|^4\right]
 \le K\big(\mathbb E[\zeta^4]+\mathbb E[\zeta^2]\big).
\]
Since $\zeta\le h\le T$, we have $\zeta^4\le h^3\zeta$ and $\zeta^2\le h\zeta$. Hence, after enlarging the constant, we have
\[
 \mathbb E\!\left[\sup_{0\le r\le\zeta}
                    |\mu r+\sigma W_r|^4\right]
 \le K_T h\,\mathbb E[\zeta],
 \]
 where $K_T = K(1+T^2)$.
 Hence, for every $\ell>0$,
\begin{equation}\label{eq:app-exit}
 \mathbb P\!\left(\sup_{r\leq\zeta}|A_r(x)-x|\geq\ell\right)
 \leq\frac{K_T h}{\ell^4}\,\mathbb E[\zeta].
\end{equation}
The factor $\E[\zeta]$ allows the displacement probability to be compared directly with the waiting-cost term below.

Fix $t<s$, write $h=s-t$ and $\rho=h\wedge S(x)$, and use dynamic
programming at time $s$. Thus the game may be played up to $\rho$, with
payment $v(s,A_h(x))$ if both players wait until $s$ without absorption;
when $s=T$, this payment is $G(A_h(x))$. Suppose for the moment that $a(s)$ and $b(s)$ are finite, as they are when $s=T$.

Consider first $x\geq b(s)+\ell$, and let the minimizing player wait until
$\rho$. For any choice $\tau\leq\rho$ of the maximizing player, the
payment is $f(t+\tau)$ if $\tau<\rho$. A larger payment at $\tau=\rho$
is possible only if the reserve hits zero or reaches time $s$ below $b(s)$.
Either event requires displacement at least $\ell$ before $\tau$, and the
extra payment is at most $\Delta_0$. Since $f'\leq-c$, the value of the
shortened game is therefore bounded by
\begin{equation}\label{eq:app-dividend-estimate}
 \begin{split}
 v(t,x)
 &\leq\sup_{\tau\in\mathcal T_\rho}
 \left\{f(t)-c\,\mathbb E[\tau]
 +\Delta_0\,\mathbb P\!\left(
       \sup_{r\leq\tau}|A_r(x)-x|\geq\ell\right)\right\}\\
 &\leq f(t)+\sup_{\tau\in\mathcal T_\rho}
 \left(-c+\frac{\Delta_0K_T h}{\ell^4}\right)\mathbb E[\tau].
 \end{split}
\end{equation}
Choose $C$ so that $\Delta_0K_T/C^4<c$, and put $\ell=Ch^{1/4}$.
The last supremum is zero, since its coefficient is negative and
$\tau=0$ is admissible. Together with $v\geq f$, this proves
$v(t,x)=f(t)$ for $x\geq b(s)+Ch^{1/4}$, and hence the bound for $b(t)$ in  \eqref{eq:app-displacement}.

For the capital injection boundary, take $0<x\leq a(s)-\ell$ and let the
maximizing player wait until $\rho$. For every minimizing stopping time
$\delta\leq\rho$, the payment is $m(t+\delta)$ unless both players reach
time $s$ without absorption. A shortfall then requires $A_h(x)>a(s)$,
so again a displacement of at least $\ell$ must occur before $\delta$.
The shortfall is at most $\Delta_0$. Using $m'\geq c$ and
\eqref{eq:app-exit} gives
\[
 v(t,x)\geq m(t)+\inf_{\delta\in\mathcal T_\rho}
       \left(c-\frac{\Delta_0K_T h}{\ell^4}\right)\mathbb E[\delta]
 \geq m(t).
\]
Since $v\leq m$, equality holds. With the same choice
$\ell=Ch^{1/4}$, this proves the bound for $a(t)$. If
$a(s)-Ch^{1/4}\leq0$, that bound follows directly from $a(t)\geq0$.

Apply the argument first with $s=T$. It proves
\eqref{eq:app-terminal-bound} and finiteness of $b(t)$ at every earlier
time. Spatial monotonicity and $f(t)<m(t)$ then imply finiteness of $a(t)$.
The argument can consequently be repeated for every $s<T$, proving
\eqref{eq:app-displacement}.
\end{proof}

\begin{lemma}\label{lem:app-continuity}
The boundaries are continuous on $[0,T)$ and satisfy
$0\leq a(t)<b(t)<\infty$. Their terminal limits are
$a(T-)=A_\star$ and $b(T-)=B_\star$.
\end{lemma}

\begin{proof}
Finiteness follows from Lemma~\ref{lem:app-displacement}. Continuity of $v(t,\cdot)$ at $x=0$, together with $v(t,0)=m(t)>f(t)$, implies $b(t)>0$. Continuity and spatial monotonicity of $v$
then imply that the contact sets at time $t<T$ are
$[0,a(t)]$ and $[b(t),\infty)$. In particular, both contact levels are
attained. The two levels cannot coincide, since otherwise their common point would
have value both $m(t)$ and $f(t)$, which is a contradiction. Thus $a(t)<b(t)$, and the continuation
interval is exactly $(a(t),b(t))$.

We first identify which discontinuities are still possible.
If $x>a(t)$, then $v(t,x)<m(t)$, and this strict inequality persists at
nearby times. Thus $a$ is upper semicontinuous. Similarly, testing
$0<x<b(t)$ shows that $b$ is lower semicontinuous. Combining these facts
with \eqref{eq:app-displacement}, applied to times $r<t$, proves left
continuity of both boundaries.

The same estimate also gives right limits. Indeed, for $t<r<s$, $b(r)\leq b(s)+C(s-r)^{1/4}$. First let $r\downarrow t$ and then take the lower limit as $s\downarrow t$.
This yields $\limsup_{r\downarrow t}b(r)\leq
\liminf_{s\downarrow t}b(s)$, so $b(t+)$ exists. Reversing the
inequalities proves existence of $a(t+)$. Both limits are finite because
$0\leq a<b$ and \eqref{eq:app-terminal-bound} bounds $b$ uniformly.
Semicontinuity leaves only
\begin{equation}\label{eq:app-possible-jumps}
 b(t)\leq b(t+),\qquad a(t)\geq a(t+).
\end{equation}
To rule out the jumps in \eqref{eq:app-possible-jumps}, we first show that $v_x(t,x)<0$ throughout the continuation region for $t<T$. Spatial monotonicity gives $v_x\leq0$, and the equation $\mathcal Lv=0$ in the continuation set implies that $v$ is
smooth there and $\mathcal L(v_x)=0$.

Fix $(t,x)$ in the continuation set, allowing $t=0$. Since
$v(t,\cdot)$ decreases from $m(t)$ to $f(t)$, the mean value theorem gives
a point $y$ in the same continuation interval with $v_x(t,y)<0$.
Choose an open interval $I$ whose closure is inside $(a(t),b(t))$ and
which contains both $x$ and $y$. Continuity of $v$ ensures that
$[t,t+h]\times\overline I$ remains in the continuation set for some
$h>0$; continuity of $v_x$ allows $h$ to be chosen so that
$v_x(t+h,\cdot)<0$ near $y$.

Let $\tau_I:=\inf\{r\geq0:A_r(x)\notin I\}$. If $t>0$, the closed
rectangle $[t,t+h]\times\overline I$ is contained in an open set where
$v_x\in C^{1,2}$ and $\mathcal L(v_x)=0$. Then It\^o's formula gives
\[
 v_x(t+h\wedge\tau_I,A_{h\wedge\tau_I}(x))-v_x(t,x) =\sigma\int_0^{h\wedge\tau_I}
             v_{xx}(t+r,A_r(x))\dd W_r.
\]
The integrand is bounded on this rectangle, so the stochastic integral
has mean zero. If $t=0$, the resulting expectation identity follows by
starting at $(\varepsilon,x)$ with elapsed horizon $h-\varepsilon$ and
then letting $\varepsilon\downarrow0$. The continuity and boundedness of $v_x$ on $[0,h]\times\overline I$, established in
Lemma~\ref{lem:app-regularity}, allow us to pass
to the limit by dominated convergence. Thus, for every $0\le t<T$, we obtain
\[
 \begin{split}
 v_x(t,x)
 &=\mathbb E\!\left[
    v_x(t+h\wedge\tau_I,A_{h\wedge\tau_I}(x))\right]\\
 &\leq\mathbb E\!\left[
    v_x(t+h,A_h(x))\,\mathbf1_{\{\tau_I>h\}}\right]<0.
 \end{split}
\]
The last inequality is strict because the nondegenerate diffusion has positive probability of remaining in $I$ until time $h$ and ending in a neighborhood of $y$ where $v_x(t+h,\cdot)<0$. This proves the claim.

Now suppose $b(t_0)<b(t_0+)$, and choose $b(t_0)<x_1<x_2<b(t_0+)$ and  $I=(x_1,x_2)$. For sufficiently small $h>0$ with $t_0+h<T$, the rectangle
$(t_0,t_0+h]\times\overline I$ lies in the continuation set. Indeed, the
right limit of $b$ keeps this interval below the dividend boundary, while
continuity of $v$ and $f(t_0)<m(t_0)$ keep the upper obstacle inactive.
At time $t_0$, however, $v(t_0,\cdot)=f(t_0)$ throughout $I$, so
$v_x(t_0,x)=0$ there.

For $x\in I$ and $0<\varepsilon<h$, the closed rectangle
$[t_0+\varepsilon,t_0+h]\times\overline I$ lies entirely in the
continuation set. We can therefore apply the same stopped It\^o formula starting at $(t_0+\varepsilon,x)$, with elapsed horizon $h-\varepsilon$, and then let $\varepsilon\downarrow0$. By the continuity and boundedness of $v_x$ on
$[t_0,t_0+h]\times\overline I$, we have
\[
 0=v_x(t_0,x)
 \leq\mathbb E\!\left[
   v_x(t_0+h,A_h(x))\,\mathbf1_{\{\tau_I>h\}}\right]<0,
\]
which is a contradiction.

If $a(t_0+)<a(t_0)$, choose
$a(t_0+)<x_1<x_2<a(t_0)$. The right limit keeps this interval above the capital
injection boundary immediately after $t_0$, and continuity of $v$ keeps
the lower obstacle inactive. Thus, the same continuation rectangle exists,
whereas $v(t_0,\cdot)=m(t_0)$ on $I$. The identical argument gives the same
contradiction. Both possible jumps in \eqref{eq:app-possible-jumps} are therefore
excluded. This proves right continuity.

It remains to prove that $a(t)\to A_\star$ and $b(t)\to B_\star$ as
$t\uparrow T$. For $0<x<B_\star$, $G(x)>f(T)$ and
$v(t,x)\to G(x)$ by Lemma~\ref{lem:app-regularity}. Hence
$v(t,x)>f(t)$ for $t$ sufficiently close to $T$, which implies
$\liminf_{t\uparrow T}b(t)\geq B_\star$. Likewise, for $x>A_\star$,
$G(x)<m(T)$, which gives $\limsup_{t\uparrow T}a(t)\leq A_\star$.
The reverse inequalities follow from \eqref{eq:app-terminal-bound}.
Thus $\lim_{t\uparrow T}a(t)=A_\star$ and
$\lim_{t\uparrow T}b(t)=B_\star$.
\end{proof}

\section{Numerical Implementation}
\label{app:numerical-details}

At an interior grid point $x_i$, the spatial generator is approximated
by
\begin{equation}
\label{eq:appendix-spatial-operator}
    (\mathcal L_h w)_i
    =
    \mu\frac{w_{i+1}-w_{i-1}}{2\Delta x}
    +\frac{\sigma^2}{2}
      \frac{w_{i+1}-2w_i+w_{i-1}}{\Delta x^2},
    \qquad i=1,\ldots,N_x-1.
\end{equation}
Equivalently, define
$A_h=I-\Delta t\,\mathcal L_h$. Its interior coefficients are
\[
    (A_h)_{i,i-1}
    =
    -\Delta t
    \left(
        \frac{\sigma^2}{2\Delta x^2}
        -\frac{\mu}{2\Delta x}
    \right),\qquad  (A_h)_{i,i}
    =
    1+\frac{\sigma^2\Delta t}{\Delta x^2},\qquad (A_h)_{i,i+1}
    =
    -\Delta t
    \left(
        \frac{\sigma^2}{2\Delta x^2}
        +\frac{\mu}{2\Delta x}
    \right).
\]
The known boundary values $v_0^n=m_n$ and $v_{N_x}^n=f_n$ are absorbed
into the right-hand side at $i=1$ and $i=N_x-1$, respectively.
The terminal condition is imposed directly at $t=T$.

For a fixed time level $t_n$, let $q^n$ denote the right-hand side after
the boundary contributions have been incorporated. Given the current
iterate $v^{n,(k)}$, define the Gauss--Seidel candidate
\[
    \widehat v_i^{\,n,(k+1)}
    =
    \frac{1}{(A_h)_{i,i}}
    \left[
        q_i^n
        -(A_h)_{i,i-1}v_{i-1}^{n,(k+1)}
        -(A_h)_{i,i+1}v_{i+1}^{n,(k)}
    \right].
\]
The projected SOR update is
\begin{equation}
\label{eq:appendix-psor}
    v_i^{n,(k+1)}
    =
    \min\left\{
        m_n,\,
        \max\left[
            f_n,\,
            v_i^{n,(k)}
            +\omega
             \bigl(
                 \widehat v_i^{\,n,(k+1)}
                 -v_i^{n,(k)}
             \bigr)
        \right]
    \right\},
\end{equation}
for $i=1,\ldots,N_x-1$. In the numerical experiments, we use $\omega=1.35$ and terminate the iteration when
\[
    \left\|
        v^{n,(k+1)}-v^{n,(k)}
    \right\|_\infty
    \leq
    \varepsilon_{\mathrm{PSOR}},
    \qquad
    \text{with } \varepsilon_{\mathrm{PSOR}}=10^{-9}.
\]
The solution from time level $t_{n+1}$ is used as the initial iterate at
$t_n$.

For spatial truncation, let $t<T$, set $H=T-t$, and suppose that $ R-x>|\mu|H$. For the uncontrolled diffusion
$A_u(x)=x+\mu u+\sigma W_u$, the probability of reaching the artificial boundary before maturity
satisfies
\[
    \mathbb P_{t,x}
    \left(
        \sup_{0\leq u\leq H}
        \{\mu u+\sigma W_u\}
        \geq R-x
    \right)
    \leq
    \mathbb P
    \left(
        \sup_{0\leq u\leq H}\sigma W_u
        \geq R-x-|\mu|H
    \right).
\]
By the reflection principle,
\begin{equation*}
    \mathbb P_{t,x}(\tau_R\leq T-t)
    \leq
    2\left[
        1-\Phi\left(
            \frac{R-x-|\mu|(T-t)}
                 {\sigma\sqrt{T-t}}
        \right)
    \right],
\end{equation*}
where $\tau_R= \inf\{u\ge 0: A_u(x)\ge R\}$ is the first hitting time of $R$ and $\Phi$ is the
standard normal distribution function.

Let  $ \Delta_0:= \sup_{0\leq s\leq T}\bigl(m(s)-f(s)\bigr)$. The resulting truncation error is bounded by
\begin{equation*}
    2\Delta_0
    \left[
        1-\Phi\left(
            \frac{R-x-|\mu|(T-t)}
                 {\sigma\sqrt{T-t}}
        \right)
    \right].
\end{equation*}
Accordingly, $R$ is selected so that this quantity is negligible over
the region in which numerical results are reported.

At each preterminal time level, we identify the contact sets using the tolerance $\varepsilon_c=2\times10^{-7}$. Hence, the definitions
\[
    a_h(t_n)
    =
    \max\left\{
        x_i:
        v_j^n\geq m_n-\varepsilon_c
        \text{ for every }j\leq i
    \right\}
\]
and
\[
    b_h(t_n)
    =
    \min\left\{
        x_i:
        v_j^n\leq f_n+\varepsilon_c
        \text{ for every }j\geq i
    \right\}
\]
select the connected contact intervals adjacent to the left and right
boundaries. If only $x_0=0$ belongs to the upper-obstacle contact set,
we set $a_h(t_n)=0$. If the upper-obstacle contact set is empty at maturity, we use $a_h(T)=0$ solely for plotting. We also set $  b_h(T)=B_\star$, consistently with the terminal dividend boundary. Since convergence of
the value approximation does not automatically imply convergence of
the extracted free boundaries, boundary robustness is assessed by
varying both $(N_x,N_t)$ and $\varepsilon_c$ in the baseline numerical results.

For the canonical examples, we reconstruct the control value using \eqref{eq:numerical-value-anchor}--\eqref{eq:numerical-value-reconstruction}. The integral $\int_{t_n}^{T}f(s)\,\dd s$  in \eqref{eq:numerical-value-anchor} is evaluated analytically whenever a closed-form expression is available. The integral $\int_{t_n}^{T}f'(s)b_h(s)\,\dd s$ is evaluated using the composite trapezoidal rule and the discrete boundary values, with piecewise linear interpolation
in time. The spatial integral in \eqref{eq:numerical-value-reconstruction} is also evaluated using the composite trapezoidal rule. If the integration endpoint is not a spatial grid point, linear interpolation is used to evaluate the contribution
from the final partial interval.

\bibliographystyle{apalike}
\bibliography{references}

\end{document}